\documentclass{amsart}
\usepackage{color}

\usepackage{amsmath} 
\usepackage{amssymb}
\usepackage{mathrsfs}

\newtheorem{theorem}{Theorem}[section] 
\newtheorem{claim}[theorem]{Claim}
\newtheorem{tmt}[theorem]{The Main Theorem}

\newtheorem{lemma}[theorem]{Lemma}

\theoremstyle{definition}
\newtheorem{definition}[theorem]{Definition}

\newtheorem{conjecture}[theorem]{Conjecture}

\theoremstyle{remark}
\newtheorem{remark}[theorem]{Remark}

\newcommand{\rk}{\mathrm{rk}}

\newcommand{\CH}{{\rm CH}}
\newcommand{\SH}{{\rm SH}}
\newcommand{\Ord}{{\rm Ord}}
\newcommand{\Hom}{{\rm Hom}}

\newcommand{\MA}{{\rm MA}}
\newcommand{\Df}{{\rm Df}}
\newcommand{\Ded}{{\rm Ded}}

\newcommand{\otp}{{\rm otp}}
\newcommand{\qd}{{\rm q.d.}}

\newcommand{\Ext}{{\rm Ext}}

\newcommand{\NPT}{{\rm NPT}}

\newcommand{\Th}{{\rm Th}}

\newcommand{\ZFC}{{\rm ZFC}}

\newcommand{\tp}{{\rm tp}}
\newcommand{\cf}{{\rm cf}}

\newcommand{\Dom}{{\rm Dom}}

\newcommand{\Aut}{{\rm Aut}}

\newcommand{\df}{{\rm df}}

\newcommand{\cov}{{\rm cov}}

\newcommand{\Rang}{{\rm Rang}}

\newcommand{\wilog}{{\rm without loss of generality}}

\newcommand{\then}{{\underline{then}}}

\newcommand{\oor}{{\underline{or}}}
\newcommand{\Then}{{\underline{Then}}}

\newcommand{\Iff}{{\underline{iff}}}
\newcommand{\mn}{{\medskip\noindent}}
\newcommand{\sn}{{\smallskip\noindent}}

\newcommand{\bfF}{\mathbf{F}}

\newcommand{\bfL}{\mathbf{L}}

\newcommand{\bfS}{\mathbf{S}}

\newcommand{\bfV}{\mathbf{V}}

\newcommand{\bbL}{\mathbb{L}}
\newcommand{\bbN}{\mathbb{N}}

\newcommand{\bbQ}{\mathbb{Q}}

\newcommand{\bbZ}{\mathbb{Z}}

\newcommand{\cH}{\mathscr{H}}

\newcommand{\cL}{\mathscr{L}}

\newcommand{\cP}{\mathscr{P}}

\newcommand{\cT}{\mathscr{T}}

\newcommand{\clP}{\mathcal{P}}

\newcommand{\gC}{\mathfrak{C}}

\newcommand{\rest}{\restriction}

\newcommand{\LL}{\langle}
\newcommand{\RR}{\rangle}

\newcount\skewfactor
\def\mathunderaccent#1#2 {\let\theaccent#1\skewfactor#2
\mathpalette\putaccentunder}
\def\putaccentunder#1#2{\oalign{$#1#2$\crcr\hidewidth
\vbox to.2ex{\hbox{$#1\skew\skewfactor\theaccent{}$}\vss}\hidewidth}}

\newbox\noforkbox \newdimen\forklinewidth
\setbox0\hbox{$\textstyle\bigcup$}
\setbox1\hbox to \wd0{\hfil\vrule width \forklinewidth depth \dp0
                        height \ht0 \hfil}
\setbox\noforkbox\hbox{\box1\box0\relax}
\def\unionstick{\mathop{\copy\noforkbox}\limits}
\def\nonfork#1#2_#3{#1\unionstick_{\textstyle #3}#2}
\def\nonforkin#1#2_#3^#4{#1\unionstick_{\textstyle #3}^{\textstyle
    #4}#2}
\setbox0\hbox{$\textstyle\bigcup$}
\setbox1\hbox to \wd0{\hfil{\sl /\/}\hfil}
\setbox2\hbox to \wd0{\hfil\vrule height \ht0 depth \dp0 width
                                \forklinewidth\hfil}
\newbox\doesforkbox
\setbox\doesforkbox\hbox{\box1\box0\relax}
\def\nunionstick{\mathop{\copy\doesforkbox}\limits}

\def\fork#1#2_#3{#1\nunionstick_{\textstyle #3}#2}
\def\forkin#1#2_#3^#4{#1\nunionstick_{\textstyle #3}^{\textstyle
    #4}#2}

\newcommand{\stickT}{%
\setbox255=\hbox{\raise1ex\hbox{$\hspace{0.2pt}\,\bullet\,$}}
\mathord{\rlap{\hbox to\wd255{\hss\hbox{$|$}\hss}}
\box255}
}
\newcommand{\stickS}{%
\setbox255=\hbox{\raise0.6ex\hbox{$\scriptstyle\bullet$}}
\mathord{\rlap{\hbox to\wd255{\hss\hbox{$\scriptstyle|$}\hss}}
\box255}
}

\newenvironment{PROOF}[2][\proofname.]
   {\begin{proof}[#1]\renewcommand{\qedsymbol}{\textsquare$_{\rm #2}$}}
   {\end{proof}}

\usepackage[hidelinks]{hyperref}

\begin{document}
\makeatletter\def\shfiuwefootnote{\gdef\@thefnmark{}\@footnotetext}\makeatother\shfiuwefootnote{Version 2022-08-08\_2. See \url{https://shelah.logic.at/papers/E40/} for possible updates.}

\title {A collection of abstracts of Shelah's Papers \\
 E40}
\author {Saharon Shelah}
\address{Einstein Institute of Mathematics\\
Edmond J. Safra Campus, Givat Ram\\
The Hebrew University of Jerusalem\\
Jerusalem, 9190401, Israel\\
 and \\
 Department of Mathematics\\
 Hill Center - Busch Campus \\ 
 Rutgers, The State University of New Jersey \\
 110 Frelinghuysen Road \\
 Piscataway, NJ 08854-8019 USA}
\email{shelah@math.huji.ac.il}
\urladdr{http://shelah.logic.at}
\thanks{The author thanks Alice Leonhardt for the beautiful typing.
  References like \cite[Th0.2=Ly5]{Sh:950} means the label of Th.0.2
  is y5.  The reader should note that the version in my website is
  usually more updated than the one in the mathematical archive.  
First typed May 1987}





\date {August 8, 2022}

\begin{abstract}
There are here abstracts of most of the papers up to 143 
(and 217), mostly written in 1980/81 with Grossberg.  Also more
details than in the originals were added in B5, B8, B217 and C2, C3
were added (the Cxx are representations of the authors works).
\end{abstract}

\maketitle
\numberwithin{equation}{section}
\newpage

\section {Notation}

\bigskip
\begin{enumerate}
    \item[(A)]
    \begin{enumerate}
        \item[(a)]   $\alpha,\beta,\gamma,\delta,\zeta,\xi,i,j$ denote ordinals; usually $\delta$ denotes a limit ordinal, $\omega$ the first infinite ordinal
\sn
        \item[(b)]  $\lambda,\chi,\mu,\kappa$ stands for infinite cardinals
\sn
        \item[(c)]  $n,m,k,l$ natural numbers
\sn
        \item[(d)]  $\lambda^\kappa$ is cardinal exponentiation, 
        $\lambda^{< \kappa} = \sum\limits_{\mu < \kappa} \lambda^\mu$, 
        $\beth(\lambda,\alpha) = \sum\limits_{\beta < \alpha} 2^{\beth(\lambda,\beta)} + \lambda$, 
        $\beth_\alpha = \beth(\aleph_0,\alpha)$
\sn
        \item[(e)] ${}^\alpha \lambda = \{f:if:\alpha \rightarrow \lambda\}$ sometimes  we call the member sequences
\sn
        \item[(f)]  ${}^{\alpha >}\lambda = \bigcup\limits_{\beta < \alpha}  {}^\beta \lambda$
\sn
        \item[(g)]  $\eta,\nu$ stands for sequences (or ordinals, usually)
\sn
        \item[(h)]  GCH is the Generalized Continuum Hypothesis
    \end{enumerate}
\sn
    \item[(B)]
    \begin{enumerate}
        \item[(a)]  $M,N$ (perhaps with index) are models
\sn
        \item[(b)]  $\Th(M)$ - the collection of first order sentences which are true in $M$.  All languages are with equality.
\sn
        \item[(c)]  a model is $\lambda$-like if it has cardinality $\lambda$ and one of the relations of the model is an order $<$ such that every initial segment has cardinality less than $\lambda$, but not the whole.
        
        \item[(d)] vocabularies are denoted by $\tau$.
    \end{enumerate}
\end{enumerate}
\newpage 

\section {abstracts} 
\bigskip

\noindent
(B1) \underline{Stable Theories}, IJM 7(1969), 186-202.

The stability spectrum (i.e. the set of cardinals in which a theory $T$ is
stable) is characterized for countable theories and investigated in the
general case; ranks for one formula are investigated, as well as existence
of indiscernibles and prime among e.g. $|T|^+$-saturated models.  Also
categoricity of elementary and of pseudo-elementary classes is investigated
and lower bound on the number of non-isomorphic models of power 
$\aleph_\alpha$ is given for unstable and unsuperstable
$T$ (they are $|\alpha - \beta|$, $|(\alpha - \beta)/\omega|$ respectively, where
$|T| = \aleph_\beta$).
\bigskip

\noindent
(B2) \underline{Note on a Min-Max Problem of Leo Moser}, J. Comb. Th. A 6(1969), 187-202

Moser asks how a pair of $(n+1)$-sided dice should be loaded (identically)
so that on throwing the dice the frequency of the most frequently occurring
sum is as small as possible.  G.F. Clements finds a relative minimum,
conjecturing that it is always the solution.  This conjecture is disproved
for $n=3$.
\bigskip

\noindent
(B3) \underline{Finite diagrams stable in power}, AML 2(1970), 69-118. 

Let $T$ be a complete first order theory, $D_m(T)$ the set of complete
$m$-type (consistent with $T$), $D(T) = \bigcup\limits_{m} D_m(T)$ and let
$D \subseteq D(T)$.  Let 

\begin{equation*}
\begin{array}{clcr}
K_D = \{M: &M\text{ is a model of } T \text{ such that if }
m < \omega \text{ and } \bar a \in {}^m M, \\
  &\text{ then the type of } \bar a \text{ (in } M \text{ of course) 
belongs to } D\}.
\end{array}
\end{equation*}

\mn
Assume $\gC_D$ is a $D$-monster, that is: it is a model of $T$, 
in $K_D$ and is a $\bar \kappa$-sequence homogeneous.  
We develop stability theory for it. 
\newpage

\centerline {Anotated content}
\bigskip

\noindent
\S1 
\mn
\begin{enumerate}
\item[${{}}$]  [Let $T$ be a complete first order theory, ${\gC}$ its
monster models $M,N$ are $\prec \gC$, $A,B$ are $\subseteq \gC$. Let
$D(T)$ be the set of complete $L(T)$-types realized in models of $T$ and let
$D$ be a subset of $D(T)$.  We call $A$ a $D$-set if every finite sequence
from $A$ realizes (in $\gC$) a type from $D$, $D(A)$ the minimal such $D$.
For a $D$-set let $\bfS^n_D(A)$ be the set of $p \in \bfS^n(A)$ 
such that if $\bar a = \langle a_\ell : \ell < n \rangle \in {}^n {\gC}$ 
realizes $p$ then $A \cup \bar a$ is a $D$-set.  
$M$ is called a $D$-model if $(M \prec {\gC}$ and) $D(M) \subseteq
D$; if $n=1$ we may omit.  We say $M$ is $(D,\lambda)$-homogeneous if (it is
a $D$-model and) $A \subseteq M$ and $|A| < \lambda$ and $p \in \bfS_D(A)
\Rightarrow p$ realized in $M$ (if $D = D(T)$ this means $M$ is
$\lambda$-saturated).  Let $D$ be fixed, all sets are $D$-sets, all models
are $D$-models.]
\end{enumerate}
\bigskip

\noindent
\S2
\mn
\begin{enumerate}
\item[${{}}$]  [We call $D$ $\lambda$-good if for every $\lambda$ there is
$(D,\lambda)$-homogeneous model of cardinality $\ge \lambda$.  
We call $D$ $\lambda$-stable if it is
$\lambda^+$-good and for every $D$-set $A$ of cardinality $\le \lambda$,
$|A| \le \lambda \Rightarrow |\bfS_D(A)| \le \lambda$.  We say $D$ is stable if
it is stable in some $\lambda$.  We say $p \in \bfS^n_D(A)$ split over
$B \subseteq A$ if $0$ for $\bar a,\bar b \subseteq A$ realizing the same
type over $B$ and $\varphi$, $\varphi(\bar x,\bar a) \equiv \varphi(\bar x,
\bar b) \in p$.  Basic properties of splitting are proved.  E.g. if
$|\bfS_D(A)| > \mu_0 = |A|^{< \lambda} + \sum\limits_{\mu < \lambda}
2^{|D|^\mu}$ then we an find $p \in \bfS_D(A)$ and an increasing sequence
$\langle A_i:i < \lambda \rangle$, $|A_i| < \aleph_0 + |i|^+$ such that
$p \rest A_{i+1}$ split over $A_i$, this in turn implies unstability
in every $\mu < 2^\lambda$.  Hence, if $D$ is unstable in every 
$\lambda < \beth \big( (2^{|T|})^+ \big)$ then $D$ is unstable (in fact have an indiscernible
sequence (which forms a $D$-set and) which ordered by some formula).]
\end{enumerate}
\bigskip

\noindent
\S3
\mn
\begin{enumerate}
\item[${{}}$]  [If $D$ is stable, then $D$ is good, and if $B \subseteq A$, 
$p \in \bfS^n_D(B)$ then $p$ can be extended to some $q \in \bfS^n_D(A)$.  
We also prove that if $D$ is stable in $\lambda$, 
$A \cup \bigcup \{\bar a_i : i < \lambda^+\}$ is a $D$-set, each $\bar a_i$ 
a finite sequence then for some $X \subseteq \lambda^+$, $|X| = \lambda^+$, 
we have $\langle \bar a_i:i \in X \rangle$ is an indiscernible set over $A$.]
\end{enumerate}
\bigskip

\noindent
\S4
\mn
\begin{enumerate}
\item[${{}}$]  [We say $p \in \bfS_D(A)$ splits strongly over 
$B \subseteq A$ if for some $C$, $\langle \bar a_i : i < \omega \rangle$ 
we have $A \cup \bigcup\limits_{i < \omega} \bar a_i \subseteq C$ 
(the point is that $C$ is a $D$-set) $\langle \bar a_i : i < \omega \rangle$ is 
an indiscernible sequence (equivalently set by \S3) over $A$ and for 
some $\psi$, $\psi(\bar x,\bar a_0),\neg \psi(\bar x,\bar a_i) \in p$.  
We investigate it and using it characterize the stability spectrum.  
I.e. if $D$ is stable, \then \, for some cardinal $\lambda = \lambda(D)$, 
$\kappa = \kappa(D)$ which are $< \beth \big(2^{|T|})^+\big)$ we 
have: $D$ is $\mu$-stable iff $\mu \ge \lambda(D)$ and 
$\mu = \mu^{< \kappa(D)}$.  Also, a formula cannot divide an indiscernible 
set to two large pieces (i.e. of size $\ge \kappa(D)$).]
\end{enumerate}
\bigskip

\noindent
\S5
\mn
\begin{enumerate}
\item[${{}}$]  [Some variants of $\lambda$-isolation are defined and so
$\lambda$-prime models over $(D)$-sets are proved to exist using what is
called in \cite[Ch.IV]{Sh:a} primary models are defined; in fact, the 
framework of it applies here (including uniqueness of primary).]
\end{enumerate}
\bigskip

\noindent
\S6  
\mn
\begin{enumerate}
\item[${{}}$]  [We prove that there is a $(D,\lambda)$-homogeneous non-$(D,\lambda^+$)-homogeneous models in all but ``degenerate" cases (parallel
to unidimensional $T$).]
\end{enumerate}
\bigskip

\noindent
\S7
\mn
\begin{enumerate}
\item[${{}}$]  [We prove the parallel of ``characterizing the categoricity
spectrum replacing categorical by all models are homogeneous, for class of
the form $\{M:M$ a model of $T$ omitting every type in a set $\Gamma\}$.]
\end{enumerate}
\mn
We prove: e.g. if $T$ is first order complete countable not categorical 
in $\aleph_2$ then it has $\ge |\alpha +1|$ non-isomorphic models in 
$\aleph_0$.

See \cite{Sh:54} on the spectrum of $\{\lambda$ : there is a 
$(D,\lambda)$-homogeneous model of cardinality $\lambda\}$.
\newpage

(B4) \underline{On Theories $T$ categorical in $|T|$}, JSL 35(1970),73-82.

We prove that, if $T$ is categorical in $|T|$, $|T| = |T|^{\aleph_0}$, then $T$
has a model of power $< |T|$; moreover, $T$ is a definitional extension of
a theory of smaller power.  We first note that $T$ has a $\aleph_1$-compact
model, and if $T$ is unsuperstable also a non $\aleph_1$-compact model of
power $\aleph_1$.  So $T$ is superstable.  We construct a model which is
``almost prime" over a countable indiscernible set, and prove it is maximal
in it, but we can easily build a model without such set.

Superseded by \cite[Ch.IX]{Sh:a}.
\bigskip

\noindent
(B5) \underline{On languages\footnote{Based on part of the author's
M.Sc.Thesis and the summary here says somewhat more on it.}
with non-homogeneous strings of quantifiers}, IJM 8(1970),75-79.

In infinitary logic we can introduce a string of quantifiers of the form
$$(\ldots Q_ix_i \ldots Q_jx_j \ldots)_{i \in I} \varphi(\bar x)$$ 
$I$ a linear order, $Q_i$ is $\exists$ or $\forall$.  For $I$ well ordered the 
interpretation is by a game, generally by Skolem functions.  The theorem is
that we can express such quantifiers by a well ordered sequence of
quantifiers (not so short; and propositial connectives).  In fact, only 
the case when $I$ an inverse of a well ordered set is presented in 
\cite{Sh:5};  the full proof appears in the author M. Sc. thesis.

(In somewhat revised formulation this was sent to a student of Takeuli 
and included in his thesis, Urbana 1973). 
\bigskip

\noindent
More formally
\begin{definition}
\label{a0}  
(Mostowski).  The meaning of 

\[
M \models (\ldots,(Q_t\bar x_t),\ldots)_{t \in I} \varphi(\ldots,
\bar x_t,\ldots)
\]

\mn
where $I$ is a linear order, 
$\bar x_t = \langle x_{t,\alpha} : \alpha < \alpha_t \rangle$, 
$Q_t \in \{\forall,\exists\}$, letting 
$$I_\exists = \{t : Q_t = \exists\},\ I_\forall = \{t \in I: Q_t = \forall\}$$ 
is that there are functions 
$F_{t,\alpha} = F_{t,\alpha}(\ldots,x_{s,\beta},\ldots)_{s \in I_\forall,\ s < t,\ \beta < 
\alpha_s}$ for $t \in I_\exists$, $\alpha < \alpha_t$ such that for every 
$a_{s,\beta} \in M$ (for $s \in I_\forall$, $\beta < \alpha_s$); letting 
$a_{t,\alpha} = F_{t,\alpha}(\ldots,a_{s,\beta},\ldots)_{s \in I_\forall,\ s < t,\ \beta < \alpha_s}$ we have $M \models \varphi[\ldots,\langle a_{t,\gamma}:
\gamma < \alpha \rangle,\ldots]_{t \in I}$.
\end{definition}

\begin{theorem}
\label{a2}
1) The formula $(\ldots,(Q_t\bar x_t)\ldots)_{t \in I}
\varphi(\ldots,\bar x_t,\ldots)$ is equivalent to the formula $\psi^*$
defined below.  

\noindent
Let
\mn
\begin{enumerate}
    \item[(A)]  ${\cT} = \{\bar J:\bar J = \langle J_i:i \le j \rangle$ is a strictly $\subseteq$-increasing continuous sequence of proper initial segments of $I$, 
    $J_0 = \varnothing\}$
\sn
    \item[(B)]  For $\bar J = \langle J_i:i \le j \rangle$ let $t \in \bar J$ mean $t \in J_j \setminus \bigcup\limits_{i < j} J_i$ and let $\bar J_1 \le \bar J_2$ mean $\bar J_1$ is an initial segment of $\bar J_2$
\sn
    \item[(C)]  For $\bar J = \langle J_i:i \le j \rangle \in \cT$ let
\[
\bar y_{\bar J} = \langle \ldots,x_{s,\beta},\ldots \rangle_{s \in \bar J,
\beta < \alpha_s,Q_s = \forall}
\]

\[
\bar z_{\bar J} = \langle \ldots,x_{s,\beta},\ldots \rangle_{s \in \bar J,
\beta < \alpha_s,Q_s = \exists}.
\]

\sn
    \item[(D)]  Lastly, define the formula $\psi^* = (\ldots,(\forall \bar y_{\bar J}) (\exists \bar z_{\bar J}),\ldots)_{\bar J \in \cT} \langle \varphi(\ldots,\bar y_{\bar J_1}, \bar z_{\bar J_1}, \ldots)_{\bar J_1 \le \bar J_2}:\bar J \in \cT \rangle$ which means: as explained below
\sn
    \item[(E)]  in the following game between the player $\exists$ and $\forall$, the player $\exists$ has a winning strategy;
\sn
    \begin{enumerate}
        \item[(a)]  in stage $i$ of a play an initial segment $J_i$ of $I$, strictly increasing with $i$ is constructed and a sequence $\langle \ldots,a^i_{s,\beta}, \ldots \rangle_{s \in I,\beta < \alpha_s}$ are constructed such that
\sn
        \begin{enumerate}
            \item[$(*)$]   $s \in J_{i_0},i_0 < i_1 \Rightarrow a^{i_0}_{s,\beta} = a^{i_1}_{s,\beta}$
        \end{enumerate}
\sn
        \item[(b)]  in stage $i$, the player $\forall$ chooses $a^i_{s,\beta}(s \in I, \beta < \alpha_s,Q_s = \forall)$ but satisfying $(*)$, \then \, player $\exists$ chooses $a^i_{s,\beta}(s \in I,\beta < \alpha_s,Q_s = \exists)$ but satisfying $(*)$, and then player $\forall$ chooses $J_j$ as required.  If in stage $\beta$ 
        $M \models \neg \varphi(\ldots,a^i_{s,\beta},\ldots)$, player $\exists$ loses immediately.  If $\exists$ never loses, he wins. 
    \end{enumerate}
\end{enumerate}
\mn
2) Formulas as in (1) (with tree quantifiers indexed by ${\cT}$ a
well founded tree) can be translated to one of the form 
$(\forall \bar y^0)(\exists \bar z^0) \ldots 
(\forall \bar y^\alpha)(\exists \bar z^\alpha)\ldots)_{\alpha < \gamma}
\bigwedge\limits_{\Upsilon \in \cT} \varphi_\Upsilon(\bar x^\Upsilon)$ where
$\gamma$ is the number of levels of the tree (so in our case, every
increasing sequence of initial segments has length $< \gamma$) and
$\varphi_\Upsilon$ the formula ``sitting" in the node $\Upsilon$ (we need
$\bar x^\Upsilon$ not to repeat variables appearing in incomparable nodes of
the tree, in our case all are the same $\varphi$).
In fact $\gamma^* \ge \otp(X)$ for every branch of the tree
${\cT}$, then $\gamma^*$ can serve and if $\alpha < \gamma^*,
\ell g(\bar y^\alpha) = \Sigma\{|\ell g(\bar y_t)|:t \in {\cT} \text{ of
level } \alpha\},\ell g(\bar z^\alpha) = \Sigma|\ell g(\bar z_t)|:t
\in {\cT} \text{ of level } \alpha\}$ are O.K.
\end{theorem}

\begin{PROOF}{\ref{a2}}
Straightforward, we may wonder: in the new formulas the ``$(\exists
\bar z^\alpha)$ may depend on additional variables but because fo the
structure of $\bigwedge\limits_{\Upsilon \in \cT} \varphi_\beta(\bar
x^\Upsilon)$ this is not a problem.
\end{PROOF}

\noindent
Another result from the M.Sc. thesis is:
\begin{theorem} 
\label{a5}
[Hanf number of $\bbL_{\lambda^+,\lambda^+}$ with unary
function only.]

Let the vocabulary $\tau$ be $\{F\} \cup \{P_i:i < \lambda\}$ where $F$ is a
unary function and the $P_i$ are unary predicates.
\mn
\begin{enumerate}
    \item[(A)]   if $\psi \in \bbL_{\lambda^+,\lambda^+}$ has models in arbitrarily large cardinality $< \beth_{\lambda^+}$ (in fact $> \beth_{\qd}(\psi)+1(\lambda)$ suffice), 
    \then \, it has models in arbitrarily large cardinalities (in fact, any one above $\beth_{\qd}(\psi)+1(\lambda)$)
\sn
    \item[(B)]  if $T \subseteq \bbL_{\lambda^+,\lambda^+}$ has a model of cardinality 
    $> \beth_{\lambda^+}$, \then \, it has models in arbitrarily large cardinalities.
\end{enumerate}
\end{theorem}

\begin{PROOF}{\ref{a5}}  For an $L$-model $M$ and $x \in M$ let 
$$M^{[x]} = (M \rest \{y:\text{for some } n < \omega,F^n(y) = x\},x)$$ 
(so $F^{M[x]}$ may be partial and $x$ serve as an individual constant) where
$F^0(x) = x$, $F^{n+1}(x) = F(F^n(x))$ and let 
$M^{\LL x \RR} = M \rest \{y : y \in \bigcup\limits_{n < \omega} M^{F^n(x)}\}$.  
Let $\mathrm{Pre}_M(x) = \{y:F^M(y) = x\}$.

\noindent
Why (A)?  Assume for simplicity that q.d.$(\psi)$ is a limit ordinal.
If $M$ is a model of $T$ of cardinality $> \beth_{\text{q.d.}(\psi)}
(\lambda)$ then one of the following occurs:
\sn
\begin{enumerate}
\item[$(\alpha)$]  the number of components (i.e. $M^{\LL x\RR},x \in M$) is
$> \beth_{\text{q.d.}(\psi)+1}$.  So for some $x$, for each 
$\beta < \text{ q.d.}(\psi)$ the set $\{M^{\LL y \RR}:M^{\LL y \RR},M^{\LL x\RR}$ satisfies
the same $L_{\infty,\lambda^+}$-formulas of quantifier depth 
$\le \beta\}$ has cardinality $> \lambda$. 
So we can add to $M$ any number of isomorphic copies of $M^{\LL x\RR}$
\sn
\item[$(\beta)$]  for some $x \in M,\beth_{\text{q.d.}(\psi)+1} <
|\text{Pre}_M(x)|$ so for some $y \in \text{ Pre}_M(x)$, for every
$\beta < \text{ q.d.}(\psi)$, we have $\{z \in \text{ Pre}_M(x):M^{[z]},
M^{[x]}$ satisfies the same $L_{\infty,\lambda^+}$ formulas of 
quantifier depth $\le \text{ q.d.}(\psi)\}$ has cardinality 
$> \beth_{\text{q.d.}(\psi)}$.  
So we can just ``increase" $\text{Pre}_M(x)$, giving more copies of
$M^{\LL  x\RR}$.
\end{enumerate}
\mn
Why (B)?  Similarly.
\end{PROOF}
\bigskip

\noindent
(B6) \quad \underline{A note on Hanf numbers}, PJM, 34(1970),539-543.

We show that for every $\xi < (2^\kappa)^+$, there is a theory $T$ and set
of type $\Gamma$ in a language of power $\kappa$, such that there is a model
of $T$ which omits every $p \in \Gamma$ of power $\lambda$ if and only if
$\lambda \le \beth_\xi$.  This completes the computation of an appropriate
Hanf number.  We also disprove a conjecture of Morley on the existence of
algebraic elements in such examples.
\bigskip

\noindent
(B7) \quad \underline{On the cardinality of ultraproduct of finite sets}, 
JSL, 35(1970), 83-84.

We prove that if $D$ is an ultrafilter and $\aleph_0 \le \lambda = 
\Pi n_i/D$, then $\lambda^{\aleph_0} = \lambda$.  The method is to 
use the amount of number theory $N = \Pi n_i/D$ inherit and 
$\aleph_1$-saturation of $N$.
\bigskip

\noindent
(B8) \quad \underline{Two cardinal compactness}, IJM 9(1971), 103-198.

Let $K$ be the class of $(\lambda,\mu)$-models, i.e. models such that
$\|M\| = \lambda$, $|P^M| = \mu$.  We prove that when $\mu^{\aleph_0} = \mu$,
$K$ is $\mu$-compact (i.e. a theory $T$ of cardinality $\le \mu$ has a
$(\lambda,\mu)$-model iff every finite subtheory has a 
$(\lambda,\mu)$-model).

In fact $\aleph_0$-compactness implies $\mu$-compactness, see the abstract
in the Notices of the AMS, \cite{Sh:E17}.  We also prove that then any 
$T$, $|T| \le \mu$ which has a $(\lambda,\mu)$-model, has a
$(\lambda,\mu)$-model in which only $\le 2^{\aleph_0} + |T|$ types are
realized (e.g. if $|T| = \aleph_0$, and the class of 
$(\lambda,\mu,\aleph_0)$-models is compact, we can get $\aleph_0$ types).
Also we can get models with many automorphisms (by a sequence of $\mu$
indiscernibles).

Similar theorems hold for $\mu$-like models, hence cardinality quantifiers,
and even for several such demands.

In the cases that $\aleph_0$-compactness holds, transfer theorems are reduced
to partition theorems with finite conclusions.

The models are like incomplete Ehrenfeucht-Mostowski ones, but the
indiscernibility is not ``full", just as much as corresponding to the
appropriate partition theorems.
\bigskip

\noindent
(B9) \quad \underline{Remark to ``local definability theory of
  Reyes"}, AML 2(1971), 441-448.

Let $\tau \subseteq \tau_1$, $P \in \tau_1 \setminus \tau$ and $T$ a theory in $\tau_1$.  For a model $M$ let 
$$\df(M) =: \big|\{P:(M,P) \text{ is a reduct of a model of } T\} \big|$$ and let $\Df(\lambda) := \sup\{\df(M)^+ : \|M\| = \lambda\}$.  

Then the following are equivalent:
\mn
\begin{enumerate}
    \item[(i)]  for no formula $\theta(\bar x,\bar y) \in \tau$, 
    $T \vdash (\exists \bar y)(\forall \bar x)\big[P(\bar x) \equiv \theta (\bar x, \bar y)\big]$ 
\sn
\item[(ii)] for some $\lambda \ge |\tau_1|$, $\Df(\lambda) > \lambda^+$,
\sn
\item[(iii)]  for every $\lambda \ge |\tau_1|$, $\Df(\lambda) \ge
\Ded^*(\lambda)$ (which is the first $\mu$ such that any tree of power 
$\lambda$, has $< \mu$ branches of a fixed height).
\end{enumerate}

\begin{conjecture}
\label{a8}
If for some $\lambda \ge |\tau_1|$, $\Df(\lambda) > \Ded(\lambda)$ then 
for every $\lambda$, $\Df(\lambda) = (2^\lambda)^+$.
\end{conjecture}
\bigskip

\noindent
(B10) \quad \underline{Stability, the f.c.p. and superstability}, 
AML 3(1971), 271-362.

We investigate in detail stable formulas, ranks of types and their
definability, the f.c.p., some syntactical properties of unstable formulas,
indiscernible sets and degrees of types in superstable theories.  There is
a list of all results connected with those properties, or whose proof use
them.

This list may still be of value.  Superseded by \cite[Ch.II]{Sh:a}.
\bigskip

\noindent
(B11) \quad \underline{On the number of non-almost 
isomorphic models of $T$ in a power},\\ PJM, 36(1971), 811-818.

Let $T$ be a first order theory.  Two models are almost isomorphic if they
are elementarily equivalent in the language $\bbL_{\infty,\omega}$.  We
investigate the number of non almost-isomorphic models of $T$ of power
$\lambda$ as a function of $\lambda$, $I(T,\lambda)$.  We prove $\mu > \lambda
\ge |T|$ and $I(T,\lambda) \le \lambda$ implies $I(T,\mu) \le I(T,\lambda)$.

We also get downward Lowenheim Skolem theorem for the corresponding variant
of a rigid model.
\bigskip

\noindent
(B12) \quad \underline{The number of non-isomorphic models of 
an unstable first-order theory},\\ IJM 4(1971), 473-487.

It is proved that if $T$ is an unstable (first-order) theory, $\lambda >
|T| + \aleph_0$, then $T$ has exactly $2^\lambda$ non-isomorphic models of
cardinality $\lambda$.  In fact we have stronger results: this is true for
pseudo-elementary classes, and for almost every $\lambda \ge |T| + \aleph_1$.

The method is contradictory orders.  Not too many Ehrenfeucht Mostowski
models built on pairwise contradictory order are isomorphic.  Many pairwise
contradictory orders are constructed, by using stationary sets.

In passing we deal with the existence of a family of $2^\lambda$ subsets of
$\lambda$, each of power $\lambda$, the intersection of any two finite, if it
does not exist, our proof works for $\lambda = |T| + \aleph_1$, if it does
exist, for some $\aleph_\alpha < \lambda$, 
$2^\lambda = 2^{|\alpha|} + 2^{\aleph_0}$ (in fact it follows that 
$\lambda < 2^{\aleph_0} \vee (\exists \mu,\alpha)$, 
$\mu < \alpha = \aleph_\alpha \le \lambda < \mu^{\aleph_0}$) 
and then we use a different proof.

Superceded by \cite[Ch.III,\S3]{Sh:300}, even better 
\cite[Ch.III,\S3]{Sh:e}.
\bigskip

\noindent
(B13) \quad \underline{Every two elementary equivalent models 
have isomorphic ultrapowers},\\ IJM, 10(1971), 224-233.

We prove that every two elementarily equivalent models have isomorphic
ultrapower.  In fact if $M,N$ are of cardinality $\le \lambda$, then there
is such an ultrafilter on $2^\lambda$.

Later the author notes that (see Stern, \cite{Str76}) the parallel theorem for
Banach spaces, and so get Craig interpolation theorem for the suitable
logic.
\bigskip

\noindent
(B14) \quad \underline{Saturation of Ultrapowers and Keisler's Order}

Superseded by \cite[Ch.VI]{Sh:a}, except \S1, which gives examples of
uncountable theories.
\bigskip

\noindent
(B15) \quad \underline{Uniqueness and characterization of prime models 
over sets for totally}\\ \underline{transcendental first-order theories}, JSL
(1972), 107-113.

If $T$ is a complete first-order totally transcendental theory then, over
every $T$-structure $A$, the prime model is unique up to isomorphism over $A$.
Moreover, $M$ is a prime model over $A$ iff:
\mn
\begin{enumerate}
\item[(a)]  every finite sequence for $M$ realizes an isolated type over
$A$, and 
\sn
\item[(b)]  there is no uncountable indiscernible set of $A$ in $M$.
Hence the uniqueness of the differential closure of a differential field of
characteristic zero follows.
\end{enumerate}
\mn
The proof is by a suitable induction on rank.
\bigskip

\noindent
(B16) \underline{A combinatorial problem: stability and order for
  models and theories in}\\ \underline{infinitary languages} PJM 41(1972).

Some infinite combinatorial problems are solved.  Their model-theoretic
representation is: if the model $M$ is unstable in $\lambda$, 
$\lambda = \lambda^{< \mu} + \sum\limits_{\kappa < \mu} 2^{2^\kappa}$, 
then in $M$ there is a set of sequences from $M$, of fixed finite length 
ordered by one formula.  We then prove that if $M$ is stable in $\lambda$ 
and has no ``large" ordered set, every set of $> \lambda$ elements contains an 
indiscernible subset.  We also prove that if $\psi \in
\bbL_{\lambda^+,\omega}$, $\varphi(\bar x,\bar y) \in \bbL_{\infty,\omega}$ 
and $\psi$ has a model in which $\psi$ orders a set of
power $\mu$ for arbitrarily large $\mu$ (or just $\mu < \beth_{(2^\kappa)^+}$
where $\psi,\varphi \in \bbL_{\kappa^+,\omega}$) \then \, $\psi$ has $2^\mu$ 
non-isomorphic models of power $\mu$ for each $\mu$.

Most model theory is superseded by \cite[Ch.I]{Sh:300} and Theorem 2.5
by \cite{Sh:222}.
\bigskip

\noindent
(B17) \quad \underline{For what filters is every reduced product
  saturated} IJM 12(1972), 23-31.

In this paper we characterize the filters $D$ such that for every sequence\\
$\langle M_i : i < I \rangle$ we have $\prod\limits_{i \in I} M_i/D$ is 
$\lambda$-saturated (where $\lambda > \aleph_0$).  
The characterization is: $D$ is $\lambda$-good, $D$ is
$\aleph_0$-incomplete and ${\cP}(I)/D$ is a $\lambda$-saturated Boolean
Algebra.  (Note that when the Boolean algebra ${\cP}(I)/D$ is 
$\lambda$-saturated it can be characterized directly; we can also 
restrict ourselves e.g. to $\Pi_n$-formulas only).
\bigskip

\noindent
(B18) \quad \underline{On Models with power-like orderings} JSL
37(1972), 247-267.

We prove here theorems of the form: if $T$ has a model $M$ in which
$P_1(M)$ is $\kappa_1$-like ordered, $P_2(M)$ is $\kappa_2$-like ordered...,
and $Q_1(M)$ is of power $\lambda_1,\dotsc,$ then $T$ has a model $N$ in
which $P_1(M)$ is $\kappa'_1$-like ordered..., $Q_1(N)$ is of power
$\lambda'_1,\ldots$  (In this article $\kappa$ is a strong limit singular
cardinal, and $\kappa'$ is a singular cardinal).  If the language has
power $\chi$, this is written $\chi:\langle \kappa_1,\kappa_2,\dotsc,|
\lambda_1,\ldots \rangle \rightarrow \langle \kappa'_1,\kappa'_2,\dotsc,|
\lambda'_1,\ldots \rangle$.

We also sometimes add the condition that $M,N$ omits some types.  The results
are seemingly the best possible, i.e. according to our knowledge about 
$n$-cardinal problems (or, more precisely, a certain variant of them,
speaking on cofinality, too).

Our method will enable us to reduce such problems to transfer problems which
do not mention $\kappa$-like orderings.  In many cases we can translate
problems of the form $\chi:\langle \kappa_1,\ldots|\lambda_1,\ldots \rangle
\rightarrow \langle \kappa'_1,\ldots|\lambda'_1,\ldots \rangle$ to problems
of the form 
$$\chi:\langle \cf(\kappa_1),\dotsc,\lambda_1,\ldots \rangle
\rightarrow \langle \cf(\kappa'_1),\dotsc,\lambda'_1,\ldots \rangle.$$

As an example of our results, concerning transfer between pairs of cardinals
we have: if $\chi \le \lambda' < \kappa',\chi \le \lambda < \kappa$,
then $\chi:\langle \kappa|\lambda \rangle \rightarrow \langle \kappa'|
\lambda' \rangle$ when at least one of the following conditions is
satisfied:
\mn
\begin{enumerate}
    \item[(A)]   $\cf(\kappa') \le \lambda'$,
\sn
    \item[(B)]  $\cf(\kappa) > \beth_\omega(\lambda)$,
\sn
    \item[(C)]  $\cf(\kappa) > \lambda$, $\cf(\kappa') = (\lambda')^+$, 
    $\lambda' = \sum\limits_{\mu < \lambda} (\lambda')^\mu$.
\end{enumerate}
\mn
In general, the transfers are affected by using incomplete 
Ehrenfeucht-Mostowski-like types obtained from certain polarized 
partition theorems.

\noindent
In addition we prove that:
\begin{theorem}
\label{a10}
There is a nice logic which is (fully) compact; stronger than first
order.  We got it adding the quantifier ``the cofinality of an order
is $\kappa$" to first order logic gives a compact logic (provably in ZFC of
course; subsequently more such logics were found).
\end{theorem}

We also suggest transfer theorems starting with finite cardinals, some
problems on them and the easily answered cases.  (Those corresponding to the
gaps $(\aleph_0),(\beth_\omega,\aleph_0)$), see p.250 and see more in
\cite{Sh:37}.

In proving the theorems or omitting types we use passing to non-well ordered
models of fragments of set theory (\S5).
\bigskip

\noindent
(B19) \quad \underline{Separability properties of almost
  disjoint families of sets}, IJM 12(1972), 207-214 (with P. Erd\H os).

We deal with familiies of subsets of $\omega$.  It is proved (in ZFC) that:
there is a strongly $n$-separable non $(n+1)$-separable maximal almost
disjoint family (of infinite subset of $\omega$), any for $n=2$, but (under
MA) not necessarily for $n=3$.  Also we have some similar results, also 
independent families and related finitary problems are discussed.

Saharon: read.
\bigskip

\noindent
(B20) \quad \underline{On power-like models for 
hyperinaccessible cardinals} JSL 37(1972), 531-537 (with J.H. Schmerl).

The main result of this paper is the following transfer theorem: if $T$ is an
elementary theory which has $\kappa$-like model where $\kappa$ is Mahlo of
order $\omega$, then $T$ has a $\lambda$-like model for each $\lambda >
\text{ card}(T)$.  This can be expressed as a transfer theorem for 
generalized quantifiers (we get compactness and omitting-type results, too).
\bigskip

\noindent
The partition theorem originally used by the second author is: 
\begin{theorem}
\label{a12}
If $f_\ell$ is an $\ell$-place function on $\kappa$,
where $\kappa$ is $(m+n)$-Mahlo then there is an $m$-Mahlo $\lambda <
\kappa$ and $A \subseteq \lambda$ unbounded, such that if 
$\alpha_1 < \ldots < \alpha_\ell \in A$, $\ell \le n$,
$f(\alpha_1,\dotsc,\alpha_\ell) < \alpha_i$ 
then the value of $f$ does not depend on $\alpha_{i+1},\dotsc,\alpha_n$.

What is sufficient (and proved there) is the version with having 
$A = \{\alpha_1 < \alpha_2 < \ldots < \alpha_n\}$ with $m=0$.
\end{theorem}
\bigskip

\noindent
(B21) \quad \underline{On problems of Moser and Hanson} Proc. Symp. in Graph Theory, Springer Lecture, No.303, 75-80 (with P. Erd\H{o}s).

The following problem is due to L. Moser.  Let $A_1,\dotsc,A_n$ be any $n$
sets.  Take the largest subfamily $A_{i_1},\dotsc,A_{i_r}$ which is
\underline{union-free}; i.e. 
$A_{i_{j_1}} \cup A_{i_{j_2}} \ne A_{i_{j_3}}$, $1 \le j_i \le r$, 
$1 \le j_2 \le r$, $1 \le j_3 \le r$, for every triple of distinct sets 
$A_{j_1},A_{j_2},A_{j_3}$.  Put $f(n) = \min(r)$, where the minimum is 
taken over all families of $n$ distinct sets.  

Determine or estimate $f(n)$.
Improving previous bounds we show

\[
\sqrt{2n} - 1 < f(n) < 2 \sqrt{n+1}.
\]

\mn
(B22) \quad \underline{A note on model complete models and generic
  models}, Proc. AMS 34(1972), 509-514.

We prove that there are many maximal model complete ($=$ generic) models,
and that there exists an (uncountable) theory with no generic models.
\bigskip

\noindent
(B23) \quad \underline{Some counterexamples in the
  partition calculus}, J. Comb. Th. A 15(1973), 167-174. (with F. Galvin)

We show that $\aleph_1 \nrightarrow [\aleph_1]^2_4$, i.e.
the pairs (2-element subsets; edges of the complete graph) of
a set of cardinality $\aleph_1$ can be colored with 4 colors so that every
uncountable subset contains pairs of every color, and $2^{\aleph_0}
\nrightarrow [2^{\aleph_0}]^2_{\aleph_0}$, i.e.  the pairs of real
numbers can be colored with $\aleph_0$ colors so that every set of reals of
cardinality $2^{\aleph_0}$ contains pairs of every color.  These results are
counterexamples to certain transfinite analogs of Ramsey's theorem.  Results
of this kind were obtained previously by Sierpinski and by Erd\H{o}s,
Hajnal and Rado.  The Erd\H{o}s-Hajnal-Rado result is much 
stronger than ours, but they used the continuum hypothesis and we do
not.  As by-products, we get an uncountable tournament with no 
uncountable transitive subtournament, and an uncountable partially 
ordered set such that every uncountable subset contains
an infinite antichain and a chain isomorphic to the rationals.
\bigskip

\noindent
(B24) \quad \underline{First order theory of permutation groups}, 
IJM 14(1973), 149-162.
\bigskip

\noindent
(B25) \quad \underline{Errata to first order, theory of 
permutation groups}, IJM 15(1973), 437-441.

We solve the problem of the elementary equivalence (definability) of the
permutation groups over cardinals $\aleph_\alpha$, by getting 
bi-interpretability results.  We show that it suffices to solve the problem
of elementary equivalence (definability) for the ordinals $\alpha$ in certain
second order logic, and this is reduced to the case of $\alpha < 
(2^{\aleph_0})^+$.  E.g. if the ordinal $((2^{\aleph_0})^+)^\omega$ 
(ordinal exponentiation) divides $\alpha_1,\alpha_2$ and
$\cf(\alpha_1)$, $\cf(\alpha_2) > 2^{\aleph_0}$ or $\cf(\alpha_1) = 
\cf(\alpha_2)$, then the permutation groups of $\aleph_{\alpha_1},
\aleph_{\alpha_2}$ are elementarily equivalent.  We also
solve a problem of Mycielski and McKenzie on embedding of free groups in
permutation groups, and discuss some weak second-order quantifiers.
\bigskip

\noindent
(B26) \quad \underline{Notes on combinatorial set theory}, IJM
14(1973), 262-277.

We shall prove some unconnected theorems: 

\noindent
1) (GCH) $\omega_{\alpha +1} \rightarrow (\omega_\alpha + \xi)^2_2$ when
$\aleph_\alpha$ is regular and $|\xi|^+ < \aleph_\alpha$. 

\noindent
2) There is a J\'onsson algebra in $\aleph_{\alpha +n}$, and
$\aleph_{\alpha +n} \rightarrow 
[\aleph_{\alpha +n}]^{n+1}_{\aleph_{\alpha +n}}$ if $2^{\aleph_\alpha} \le
\aleph_{\alpha +n}$. 

\noindent
3) [Universal graph]  If $\lambda > \aleph_0$ is a strong limit
cardinal, then among the graphs with $\le \lambda$ vertices each of 
valence $< \lambda$ there is a universal one. 

\noindent
4) (GCH) If $f$ is a set mapping on $\omega_{\alpha +1}$ ($\aleph_\alpha$
regular) satisfying $\alpha < \beta < \lambda \Rightarrow |f(\alpha) \cap
f(\beta)| < \aleph_\alpha$, then there is a free subset of order-type $\xi$
for every $\xi < \omega_{\alpha +1}$.
\bigskip

\noindent
(B27) \quad \underline{Size direction games over the real line. III.}, 
IJM, 14(1973), 442-449 (with G. Moran).
\bigskip

\noindent
(B28) \quad \underline{There are just four second-order quantifiers}, 
IJM, 14(1973), 262-277.

Among the second-order quantifiers ranging over relations satisfying a
first-order sentence, there are four for which any other one is
bi-interpretable with one of them: the trivial, monadic, permutational and
a full second-order.  (We deal with infinite, one sorted models).  The
interpretation is in fact first-order (second-order parameters but not
quantifiers are used).

See representation in Handbook of model theoretic logics, Baldwin's article
and \cite{Sh:171}.
\bigskip

\noindent
(B29) \quad \underline{A substitute for Hall's theorem for families 
with infinite sets}, J. Comb. Th. A. 16(1974), 199-208.

A sufficient condition for the existence of a system of distinct
representatives for a family $S$ is that $x \in A \in S$ implies 
the number of elements of $A$ is not smaller than the number of sets 
in $S$ to which $x$ belongs.

The method is: by Hall's theorem we reduce the problem for countable $S$
to the divergence of an infinite sum, which is then proved to diverge.
For $S$ uncountable see \cite{Sh:35}.
\bigskip

\noindent
(B30) \quad \underline{The cardinals of simple models 
for universal theories}, Proc. of the Symp. in honor of 
Tarski's 70th birthday, Proc. Symp. in Pure Math 25(1974), 
53-74 (with R. McKenzie)

Some results about spectra of cardinals of simple algebras in varieties and,
more generally, about the cardinals of $T$-simple models where $T$ is a
universal theory are obtained and applications discussed.  $M$ is $T$-simple 
if every homomorphism from $M$ to a model $N$ of $T$ is an embedding or a
constant.  It is shown that if the language of $T$ has power $\kappa$ and if
there exists a $T$-simple model whose power exceeds $2^\kappa$, then 
$T$-simple models exist in all powers $\lambda \ge \kappa$ (the point is that
2-indiscernibility is enough, even a sequence of approximations).  
It is further shown that if the language of $T$ is countable, and 
if there exists an uncountable $T$-simple model, then
there exists a $T$-simple model with the power $2^{\aleph_0}$.  Also
counterexamples are given showing the results are essentially best
possible.

There are generalizations.
\bigskip

\noindent
(B33) \quad \underline{The Hanf number of omitting complete types}, 
PJM 50(1974).

It is proved in this paper that the Hanf number $m^c$ of omitting complete
types by models of complete countable theories is the same as that of
omitting not necessarily complete type by models of a countable theory.  I.e.
\begin{theorem}
\label{a14}
For every ordinal $\alpha < \omega_1$ there is a countable
first-order vocabulary $\tau_\alpha$ and a complete theory $T_\alpha$ in 
$\bbL(\tau_\alpha)$ such
that:
\mn
\begin{enumerate}
\item[(i)]   $p = \{P(\chi_0)\} \cup \{\chi_0 \ne c_n:n < \omega\}$
is a complete type for $T_\alpha$
\sn
\item[(ii)]  $T_\alpha$ has a model of cardinality $\beth_\alpha$ 
omitting $p$
\sn
\item[(iii)]  $T_\alpha$ has no model of cardinality $> \beth_\alpha$
omitting $p$.
\end{enumerate}
\end{theorem}
\bigskip

\noindent
(B34) \quad \underline{Weak definability in infinitary languages}, 
JSL 38(1973), 339-404.

We prove that if a model of cardinality $\kappa$ can be expanded to a model
of a sentence $\psi$ from $\bbL_{\lambda^+,\omega}$ by adding a suitable
predicate in more than $\kappa$ ways, then it has a submodel of power $\mu$
which can be expanded to a model of $\psi$ in $> \mu$ ways provided that
$\lambda,\kappa,\mu$ satisfy suitable conditions.
\begin{tmt}
\label{a16}
1) Let $\psi$ be a sentence in $\bbL_{\lambda^+,\omega}(\tau + P)$, $P$ a unary 
predicate not in $\tau$, $|\tau| \le \lambda$, $M$ a $\tau$-model of cardinality 
$\aleph_{\alpha + \beta}$ 
such that:

\[
\big| \{P:P \subseteq |M|, (M,P) \models \psi\} \big| > \aleph_{\alpha + \beta}.
\]

\mn
Assume further that $\beta < \omega_1$, $\aleph_\alpha$ has cofinality
$\aleph_0$, $\mu_n \ge \lambda$, $\mu = \sum\limits_{n < \omega} \mu_n$, 
$\mu_n < \mu_{n+1}$ and 
$\kappa < \aleph_\alpha \Rightarrow \kappa^{\mu_n} < \aleph_\alpha$ 
for $n < \omega$.

\Then \, $M$ has an elementary submodel $N$ of cardinality $\mu$ such that

\[
\big| \big\{ P : P \subseteq |N|,\ (N,P) \models \psi  \big\} \big| \ge \mu^{\aleph_0}.
\]
\end{tmt}
\bigskip

Another theorem which we shall not prove, as its proof is simpler is
\begin{theorem}
\label{a18}
Let $\psi \in \bbL_{\lambda^+,\omega}(\tau+P)$, $M$ a $\tau$-model of cardinality
$\kappa$ such that $|\{P:P \subseteq |M|,(M,P) \models \psi\}| >
\kappa$.   Assume further that $\mu \ge \lambda$, $\kappa^\mu = \kappa$.
Then $M$ has an elementary submodel $N$ of cardinality $\mu$ such that
$$\big| \big\{P:P \subseteq |N|,\ (N,P) \models \psi \big\} \big| > \mu^{\aleph_0}.$$
\end{theorem}

Also, (see middle of page 400 there)

\begin{theorem}
\label{a20}
If $T$ is a complete (first order) theory, $|T| = \lambda^+$, $\lambda$ regular 
(for simplicity) and every $n$-type of cardinality $< \lambda$ can
be extended to complete $n$-type of cardinality $< \lambda$ (holds if
$|D(T)| < 2^\lambda$), then $T$ has a model in which every finite 
sequence realizes a complete type of cardinality $< \lambda$.
\end{theorem}
\bigskip

\noindent
(B35) \quad \underline{Sufficiency conditions for the existence of 
transversals} Cand. J. Math 26 (1974), 948-961 (with E.C. Milner).

The main theorem has an interesting formulation in terms of bipartite
graphs.  A bipartite graph is a triple $\Gamma = \langle X,\Delta,Y \rangle$
with vertex set $X \cup Y$ ($X,Y$ disjoint sets) and edge set
$\Delta \subseteq \{\{x,y\}:x \in X,y \in Y\}$.   
Let $v(z) =: \big|\{u \in X \cup Y:\{u,z\} \in \Delta\} \big|$ 
(for $z \in X \cup Y$) be the valency function of
$\Gamma$.  Then the main theorem is equivalent to the following statement:
if $\Gamma = \langle X,\Delta,Y \rangle$ is a bipartite graph such that
$v(x) > 0$ for $x \in X$ and $v(x) \ge v(y)$ whenever $x \in X,y \in Y$ and
$\{x,y\} \in \Delta$, then there is a matching from $X$ into $Y$, i.e. there
is a 1-1 function $\varphi:X \rightarrow Y$ such that
$\{x,\varphi(x)\} \in \Delta$ (for $x \in X$).

In \S7 we prove even stronger results.
\bigskip

\noindent
(B36) \quad \underline{Remarks on Cardinal invariants in Topology}, 
General Topology and its Applications 7(1977), 251-259.

In \S2 we investigate to what degree singular calibres are preserved by
products (for regular calibres Sanin solved the problem).  We also generalize
the concept to pairs of cardinals.
\bigskip

\noindent
We prove that
\begin{theorem}\label{a22}  
    If $\lambda$ is a calibre of $X,\cf(\lambda) > \aleph_0$, then $\lambda$ is a calibre of $X^I$.
\end{theorem}

In \S1 we construct a $T_2$ space $X$ with $d(X) = 2^{\aleph_0}$ and
$s(x) = \aleph_0$; i.e. there is no discrete subspace of power $\aleph_1$.
We also show that the upper bound in the definition of $h(X) = \lambda$ is
realized if $X$ is Hausdorff and $\cf(\lambda) = \aleph_0$.  The theorem on
$h(X) = \lambda$, for $X$ Hausdorff was proved independently by Hajnal and
Juhasz and the author.

In \S3 we show that the class of calibres of a $T_2$ space is just about
arbitrary.
\bigskip

\noindent
(B37) \quad \underline{A two-cardinal theorem}, Proc. AMS 481(1975), 207-213.

We prove the following theorem and deal with some related questions: if
for all $n < \omega$, $T$ has a model $M$ such that $n + |Q^M|^n \le
|P^M| < \aleph_0$ then for all $\lambda,\mu$ such that
$|T| \le \mu \le \lambda < \Ded^*(\mu)$ (e.g. $\mu = \aleph_0,
\lambda = 2^{\aleph_0}$), $T$ has a model of type $(\lambda,\mu)$, i.e.
$|Q^M| = \mu$, $|P^M| = \lambda$.  We use the existence of free subsets.
\bigskip

\noindent
\textbf{Question}:\label{b37.1}  Is our result the best possible?  
That is, does there exist a sentence which for every $n$ has 
a model $M$, $\aleph_0 > |P^M| > |Q^M|^n$, $|Q^M| \ge n$, but does 
not have a $(2^\mu,\mu)$-model for some $\mu$, and even: has a 
$(\lambda,\mu)$-model iff $\mu \le \lambda < \Ded^*(\mu)$ 
(assuming for some $\mu,\Ded^*(\mu) \le 2^\mu$).
\bigskip

\noindent
\textbf{Conjecture}:\label{b37.2}  $\{(m_i,n_i,k_i):i < \omega\} 
\rightarrow (\lambda,\mu,k)$ when $m_i \ge n^i_i$, $n_i \ge k^i_i$, 
$k_i \ge i$, $k \le \mu \le \lambda < \Ded^*k$.
\bigskip

\noindent
\textbf{Conjecture}:\label{b37.3}  $\{(2^{n_i},n_i):i < \omega\} 
\rightarrow (2^\mu,\mu)[n_i \ge i]$.  (Some information is suggested).
\bigskip

\noindent
(B38) \quad \underline{Graphs with prescribed asymmetry and minimal 
number of edges}, Erd\H os Symp. (infinite and finite sets), 
Vol.III(1975), 1241-1256.

We shall deal with non-directed graphs, without loops and double edges, and
having a finite number of vertices.

A graph is symmetric if it has a non-trivial automorphism = a permutation of
its vertices, such that a pair of vertices is connected iff their images are
connected.  The asymmetry of a graph is the minimal number of changes (i.e.
adding and deleting of edges) which is necessary to make the graph symmetric.
Erd\H{o}s and Renyi \cite{ErdRen63} 
defined and investigated this notion, and defined,
$F(n,k) \, [C(n,k)]$ for $k \ge 1,n > 1$ as the minimal number of edges in a
[connected] graph with $n$ vertices which is asymmetric is $k$; if there is no
such graph the value of the function will be $\infty$ (if $n$ is too small,
this happens).
\bigskip

\begin{theorem}
\label{0.1}
For $n$ sufficiently large

\[
F(n,2) = n+1,\ C(n,2) = n+2.
\]
\end{theorem}

\begin{theorem}
\label{0.2}  
For odd $k>2$ and $n$ sufficiently larger than $k$

\[
F(n,k) = C(n,k) = [(k+3)n/4] - 0.5[2n/(k+3)+ 1/2].
\]
\end{theorem}

\begin{theorem}
\label{0.3}
For even $k>2$ and $n$ sufficiently larger than $k$

\[
F(n,k) = C(n,k) = [(k+2)n/4 + 1/2].
\]

\mn
The proof is given only for $k \ge 41$, and uses a detailed analysis of
examples.
\end{theorem}
\bigskip

\noindent
(B39) \quad \underline{Differentially closed fields}, IJM 16(1973), 314-328.

We prove that even the prime, differentially closed field of characteristic
zero, is not minimal; that over every differential radical field of
characteristic $p$, there is a closed prime one, and that the theory of
closed differential radical fields of characteristic $p(> 0)$ is
stable (hence the prime model above is unique).
\bigskip

\noindent
(B40) \quad \underline{Notes on Partition Calculus}, Erd\H{o}s Symp. 
Vol.III, 1257-1276.

We deal here with some separate problems which appeared in the problem list
of Erd\H{o}s and Hajnal \cite{EH}.

In \S1 we solve problem 3 of \cite{EH} affirmatively.  Thus, if
$\aleph_\omega < 2^{\aleph_{n(0)}} < \ldots < 2^{\aleph_{n(k)}} < \ldots$,
then $\sum\limits_{n < \omega} 2^{\aleph_n} \rightarrow (\aleph_\omega,
\aleph_\omega)^2$.  We prove a canonization lemma for it.

In \S2 we deal with problem 32 of \cite{EH} which asks whether there is a graph
$G$ with $\aleph_1$ vertices and with no $[[\aleph_0,\aleph_1]]$ subgraph
for which $\aleph_1 \rightarrow (G,G)^2$.  We provide a wide class of such
graphs, assuming CH.  If $\bfV = \bfL$ is assumed we show that
$\aleph_1 \rightarrow (G,G)^2$ iff $G$ has coloring number $\le \aleph_0$.

In \S3 we deal with problems 48, 50 of \cite{EH} (asked by Erd\H{o}s and
Hajnal): partition relations concerning coloring numbers.

In \S4 we deal with problem 42 of \cite{EH}.  We get compactness and incompactness
results on the existence of transversals and on property $B$.

We also find sufficient and necessary conditions for the existence of
transversals for a family of sets.
\bigskip

\noindent
(B41) \quad \underline{Some theorems on transversals} 
 Erd\H{o}s Symp., Vol.III, 115-1126 (with E.C. Milner).

We prove e.g. $\NPT(\lambda,\aleph_0) \rightarrow \NPT
(\lambda^+,\aleph_0)$ where $P(\lambda)$:
there is a family $F$ of $\lambda$ countable sets such that $F$ 
has no transversals but every $F' \subseteq F$, $|F'| < |F|$ has a 
transversal.
\bigskip

\noindent
(B42) \quad \underline{The monadic theory of order}, Annals of 
Math 102(1975), 379-419.

First, a decidability method is developed and used to decide in a
uniform and simplified way all known decidable fragments of the monadic
theory of order: the monadic theory of $\omega$, the monadic theory of
$\omega_1$, the monadic theory of countable orders and so on.  The method
is used to reduce the monadic theory of ordinals $\alpha < \lambda^+$ to
that of $\lambda$ and to reduce the monadic theory of $\omega_2$ to the
first-order theory of the Boolean algebra of subsets of $\omega_2$ module
closed unbounded sets with an operation $F(X) = \{\alpha:\cf(\alpha) =
\omega_1$ and $X \cap \alpha$ is stationary in $\alpha\}$ and similarly for
any cardinal.  The new proof of decidability of the monadic theory 
of countable orders helps to prove that
countable orders cannot be characterized in the monadic logic.  To develop
the method a useful Feferman-Vaught type theorem for the monadic theory of
order is proven and also two Ramsey type theorems are proven.
\bigskip

\noindent
E.g. 
\begin{theorem}\label{b42}
1) If $c$ is an additive two-place function form a dense
linear order (i.e. if $x < y < z$, $f(x,z)$ is determined by 
$f(x,y),f(y,z)$) and $f$ has a finite range, then there is a
$c$-homogeneous set which is dense in some interval (this is made to help
prove decidability). 

\noindent
2) If $\lambda$ is regular, $c$ and additive two-place function 
from $\lambda$ with range of cardinality $< \lambda$, \then \,
there is a $c$-homogeneous set which is unbounded in $\lambda$.
\end{theorem}

Second, the true first-order arithmetic is interpreted under the Continuum
Hypothesis in the monadic theory of the real line and in the monadic theory
of order.

There are more results and many conjectures.  The introduction surveys
preceding results.
\bigskip

\noindent
(B43) \quad \underline{Generalized quantifiers and compact logic}, Trans. 
AMS 204(1975), 342-364.

We show the existence of a logic stronger than first-order logic even for
countable models, but still satisfying the general compactness theorem,
assuming e.g. the existence of a weakly compact cardinal.  We also discuss
several kinds of generalized quantifiers.

We prove the compactness of $\bbL(Q^{\cf}_{C_1},\dotsc,
Q^{\cf}_{C_n})$ when $C_n$ is a convex class of cardinals,
$Q^{\cf}_C$ say the cofinality of an order belongs to $C$, and
axiomatize $L(Q^{\text{cf}}_{C_1})$.  We also prove that 
$\aleph_0$-compactness of $(Q^{\mathrm{stat}}_{\aleph_1}x)$ (being stationary
for order of cofinality $\aleph_1$) and transfer theorems.

We define generalized second-order quantifiers: allowing in forming formulas
the use of second-order variables, and second-order quantifiers 
$(QP)\psi(P)$, such that its satisfaction depends only on the 
isomorphism type of $\{P :\ \models \psi[P]\}$.  In particular, 
stationary logic (that is, the quantifier $(\mathrm{aa}X)$) is introduced 
(developed subsequently by Barwise, Kaufman and Makkai, \cite{BKM78}) and we
note we could have used diamonds instead of a weakly compact 
cardinal in the main theorem.
\bigskip

\noindent
(B44) \quad \underline{Infinite abelian groups, Whitehead problem 
and some constructions}, IJM, 18(1974), 243-256.

We deal with three problems: the existence of indecomposable and even rigid
(abelian) groups, the number of reduced separable $p$-groups and Whitehead
problem.

We prove that for \underline{any} cardinal $\lambda > \aleph_0$ there are
$2^\lambda$ (abelian) groups $G_i$ ($i < 2^\lambda$) such that any homomorphism
$h:G_i \rightarrow G_j$ is zero \oor \, $i = j$, $h(x) = nx$.  We do it by
coding deepness of well-founded trees and stationary sets.

For regular $\lambda > \aleph_0$ we build a family of $2^\lambda$ reduced
separable $p$-groups, each of power $\lambda$, no one embeddable into 
another,
by imitating the proof for non-superstable theories (see [1], Ch.VIII and
\cite{Sh:e} for singular $\lambda$).  The proof indicates that reduced 
separable $p$-groups cannot be characterized by reasonable invariants.

An abelian group $G$ is $W$ (= Whitehead) if $\Ext(G,\bbZ) = 0$, or
equivalently, when: if $\bbZ \subseteq H$ ($\bbZ$ the integers), 
$h:H \rightarrow G$ a homomorphism onto $G$ with kernel $\bbZ$, 
then for some homomorphism $g:G \rightarrow H$ we have $hg = 1_G$.  
It is proved that ``every Whitehead 
group of power $\aleph_1$ is free", is independent of $\ZFC$: it 
follows from 
$\bfV=\bfL$ (or even $\diamondsuit_S$ for every stationary $S \subseteq \omega_1$) 
and is contradicted by $\MA + 2^{\aleph_0} > \aleph_1$.  We partition the
$\aleph_1$-free groups of power $\aleph_1$ to three classes: I (``unstable
in $\aleph_0$"), II (intermediate) and III (the free ones).  Class II is 
similar to ``strongly $\aleph_1$-free but not free" but somewhat wider.

For latter use let $\Gamma(G) = \{i:G/G_i$ is not $|G|$-free$\}$, where
$$G = \bigcup\limits_{i < |G|} G_i,$$ $G_i$ increasing 
continuums, $|G_i| < G$, and $G$ is $|G|$-free ($\Gamma(G)$ is determined modulo the filter of closed unbounded subsets of $|G|$).
\bigskip

\noindent
(B45) \quad \underline{Existence of rigid-like families 
of abelian $p$-groups, model theory}\\ \underline{and algebra}, lecture notes in 
Math. 498, Springer-Verlag, 1975, 385-402.

We prove that for arbitrarily large $\lambda$ there is a 
family of $2^\lambda$
(abelian) separable $p$-groups, each of power $\lambda$ with only the
necessary homomorphism (called simple) between them.  A homomorphism $h:G
\rightarrow H$ is zero-like if there are no $m < \omega$ and $a_n \in G$ of
order $p^{m+m}$, $h(a)$ having order $\ge p^n$; we call $h$ simple if
$h = h_1 + h_2$, $h_1$ zero-like, $h_2$ a multiplication by a $p$-adic integer
(if $H \ne G,h_2 = 0$).

We prove this for $\lambda$ strong limit of 
uncountable cofinality; the method
is that we can for each strong limit $\mu < \lambda$, $\cf(\mu) = \aleph_0$,
diagonalize over all approximation of a homomorphism; we can do it as
$2^\mu = \mu^{\aleph_0}$.

For $\lambda^{\aleph_0} = 2^\lambda > 2^{\aleph_0}$ we prove a weaker
theorem
\mn
\begin{enumerate}
\item[$(*)$]  i.e. $\forall x \bigl[\bigvee\limits_{n < \omega} p^nx = 
0 \bigr]$, and the group has no divisible subgroup $\ne 0$.
\end{enumerate}
\bigskip

\noindent
(B46) \quad \underline{Colouring without triangles and 
partition relation}, IJM 20(1975), 1-12.

We say the \underline{colouring} $h$ of $\lambda$ proves $\lambda \rightarrow
[\mu]^2_\chi$ if $h$ is a function from the edges of the complete graph on
$\lambda$ into $\chi$ such that whenever $S \subseteq \lambda$, $|S| = \mu$, 
then the range of $h$ restricted to the edges in $S$ is the whole of $\chi$.

Erd\H{o}s and Hajnal (Problem 68) asked if it is true (assuming CH) that
for every colouring which witness $\aleph_1 \rightarrow [\aleph_1]^2_3$ there
is a triangle whose edges have three colours.  We answer this negatively
and prove the stronger result (Theorem 1.1) that, if
$2^{\aleph_\alpha} = \aleph_{\alpha +1}$, then there is a colouring which
witness $\aleph_{\alpha +1} \rightarrow 
[\aleph_{\alpha +1}]^2_{\aleph_\alpha}$and for which there 
is no triangle with three colours.  We also show that, if $\bfV = \bfL$, and 
$\aleph_\alpha$ is regular then there is even such a colouring
which witness $\aleph_{\alpha +1} \rightarrow [\aleph_{\alpha +1}]^2
_{\aleph_{\alpha +1}}$.

In \S2, we deal with the corresponding finite problem and confirm a 
conjecture of Erd\H{o}s.  We prove, for example, that if the edges of the
complete graph on $n$ points are colored with three colours so that no 
triangle has three colours, then there is a set $A \subseteq n$, such that
$|A| \ge n^c$ ($c$ some constant) and which contains only two colours.  This
should be compared with the result $n \nrightarrow [c_1 \log(n),n]^2_3$
proved by Erd\H{o}s by a probabilistic method.

In \S3, we show that examples for 1.1 should contain quite 
homogeneous sets; and then introduce new partition relations, (of interest
in the finite and infinite case), give some results and suggest some
problems.

It was proved by Erd\H{o}s and Hajnal that any example of
$\aleph_1 \rightarrow \bigl[ [\aleph_1,\aleph_0] \bigr]^2_3$ is
$\aleph_0$-universal.  Erd\H{o}s and Hajnal asked whether every example
of $\aleph_2 \rightarrow \bigl[ [\aleph_2,\aleph_1] \bigr]^2_3$ is
$\aleph_1$-universal.  We prove the consistency of the negative answer and
the method is applicable in many other cases.
\bigskip

\noindent
(B47) \quad \underline{$\Delta$-Logics and generalized 
quantifiers}, AML 10(1976), 155-192 (with J.A. Makowsky and Y. Stavi).

We study the $\delta$-closure of an abstract logic, which is its smallest
extension satisfying the $\Delta$-interpolation theorem or, equivalently,
its largest extension having the same $PC$-classes.

We study properties which are preserved by the $\Delta$-closure such as
compactness, and various Lowenheim and Hauf numbers.

We apply these to exhibit new logics with various properties, and we study
$\Delta$-closed fragments of $\bbL_{\omega_1,\omega}$ in great detail.
\bigskip

\noindent
(B48) \quad \underline{Categoricity in $\aleph_1$ of 
sentences in $\bbL_{\omega_1,\omega}(Q)$}, IJM 20(1975), 127-148.

We investigate the possible numbers of models of a sentence 
$\psi \in \bbL_{\omega_1,\omega}(Q)$ of cardinality $\aleph_1$.

Assume $2^{\aleph_0} < 2^{\aleph_1}$ and $1 \le I(\aleph_1,\psi) <
2^{\aleph_1}$, it is proved that there exists a class $K$ of models which
are the atomic models of some first-order countable theory $T_\psi$ such
that $\forall \lambda[I(\lambda,K) \le I(\lambda,\psi)]$ and 
$I(\aleph_1,K) \ge 1$.

Under the same assumptions it is proved that $K$ is $\aleph_0$-stable.
\smallskip

\noindent
The Main Theorem is:  assuming $\diamondsuit_{\aleph_1}$ if $1 \le 
I(\aleph_1,\psi) < 2^{\aleph_1}$, then there exists a model of $\psi$ 
of cardinality $\aleph_2$.  This is proved by showing:
\mn
\begin{itemize}
\item  if the $\aleph_0$-amalgamation property fails for $K$, then using
$\diamondsuit_{\aleph_1}$ we can construct $2^{\aleph_1}$ non-isomorphic
models in $K$ of cardinality $\aleph_1$.
\sn
\item  The $\aleph_0$-amalgamation property is equivalent to a suitable rank
being always $< \infty$ (and some other properties).
\sn
\item  Assuming the rank is $< \infty$, we prove that the order property
(= there are $\aleph_1$ sequences in a model of $\psi$, ordered by some
formula) is equivalent to two variants of the symmetry property (one of
them: 
\newline
$R(\tp(\bar a,M \cup \bar b)) = R(\tp(\bar a,M))$ \Iff \,
$R(\tp(\bar b,M,\bar a)) = R(\tp(\bar b,M))$ essentially).
The second asymmetry property, enables us to prove 
$I(\aleph_1,K) = 2^{\aleph_1}$ by coding stationary subsets of $\omega_1$.
But the holding of the symmetry property enable us to properly extend every
model of power $\aleph_1$.  Thus by iteration, we get a model of power
$\aleph_2$.
\end{itemize}
\bigskip

\begin{remark}
1) We continue this work on $\bbL_{\omega_1,\omega}$ sentences
in [Sh:87a], [Sh:87b] where we improve some of the theorems from this paper 
and some of the claims here are preparation for them and also 
in \cite{Sh:88}.

\noindent
2) The main difficulty is in 3).
\end{remark}
\bigskip

\noindent
(B49) \quad \underline{A two-cardinal theorem and a combinatorial 
theorem}, Proc. AMS 62(1977), 134-136.

We prove a new two-cardinal theorem, e.g.
$(\aleph_\omega,\aleph_0) \rightarrow (2^{\aleph_0},\aleph_0)$.

For this we prove a new partition theorem, with ``tree indiscernibility".

The problem seemed to have been the right induction.  Also compactness is
gotten.
\begin{conjecture}\label{b49}
A) $(\aleph_{\alpha + \omega + \omega},
\aleph_{\alpha + \omega},\aleph_\alpha) \rightarrow (\lambda,\mu,\chi)$
whenever $\chi < \mu < \lambda < \mathrm{Ded}^* \chi$. 

\noindent
B)  If a countable theory $T$ has a $\lambda$-like model, $\lambda$ a limit
cardinal, and $|T| \le \mu < \lambda_1 \Ded^*(\mu)$, $\lambda_1$ a
singular cardinal, then $T$ has a $\lambda_1$-like model.  If $\lambda$ is
$\omega$-Mahlo weakly inaccessible cardinal, we can remove the cardinal, we
can remove the singularity of $\lambda_1$. 

\noindent
C) If $\psi \in L_{\omega_1,\omega}$ has a model of cardinality 
$\aleph_{\omega_1}$, then $\psi$ has a model of cardinality $2^{\aleph_0}$
(but see [Sh:522]).
\end{conjecture}
\bigskip

\noindent
(B50) \quad \underline{Decomposing uncountable squares to 
countably many chains}, J. Comb. Th. A. 21(1976), 110-114.

We construct an ordered set $I$ of cardinality $\aleph_1$, such that its
square is the union of $\aleph_0$ chains (in the natural partial order).
This gives a complete order, isomorphic to any open interval of itself but
not to its universe.  This is proved by constructing a ``very special"
Aronszajn tree.
\begin{conjecture}\label{b50}
It is consistent that:
\mn
\begin{itemize}
\item  any Specker order contains a suborder as above
\sn
\item  if $I,J$ are as above, $I,J$ or $I,J^*$ have uncountable isomorphic 
suborders. 
\end{itemize}
\mn
The consistency of 2) is proved in \cite{Sh:114}.

As for (1) it is equivalent to:
\mn
\begin{enumerate}
\item[$(*)$]  if $T$ is an Aronszajn tree, $c:T \rightarrow 2$ then for
some uncountable $X \subseteq T$ and $\ell \ne \{0,1\}$ for every $x \ne y$
in $X$, the maximal $z \le_T x,y$ satisfies $c(z) = \ell$.
\end{enumerate}
\end{conjecture}
\bigskip

\noindent
(B51) \quad \underline{Why there are many non-isomorphic models 
for unsuperstable theories},\\ Proc. of the International Congress 
of Math. Vancouver (1974), 553-557.

We review the proof (see \cite{Sh:a}) that for unsuperstable $T$, 
$T \subseteq T_1$, $I(\lambda,T_1,T) = 2^\lambda$ when e.g. $\lambda > |T_1|$
\underline{or} $\lambda = |T_1| + \aleph_1$, $T$ unsuperstable.  To exemplify 
the generality of the method, we use it to prove that in any uncountable cardinality 
$\lambda$ there is a rigid Boolean algebra (for $\lambda$ regular satisfying 
the countable chain condition), and a rigid order (with a rigid completion).
\bigskip

\noindent
(B52) \quad \underline{A compactness theorem in singular 
cardinals, free algebras, Whitehead}\\ \underline{problem and transversals}, 
IJM 21(1975), 319-349.

We prove in an axiomatic way a compactness theorem for singular cardinals,
concluding that if $\bfV = \bfL$ every Whitehead group is free.  For this we use 
index sets of submodels $\langle M_u:u \in \clP(n) \rangle$ and 
prove various assertion by induction on the cardinality for all $n$ at once.

There is an application for transversals, i.e. if $\kappa < \lambda$, $\lambda$
singular $\{A_i: < \lambda\}$ a family of $\lambda$ sets each of power
$\le \kappa$, and every subfamily of smaller cardinality has a transversal
then the family has a transversal.

Similar conclusions hold for freeness of groups, freeness abelian groups
and colouring numbers.  

The general theorem:
\mn
\begin{enumerate}
\item[$(*)$]  ``if $V$ is a variety with $< \lambda$ functions, $A \in V$
an algebra of power $\lambda$, and every (or ``many") subalgebras of $A$ of
cardinality $< \lambda$ are free, then $A$ is free"
\end{enumerate}
\mn
and does not follow by the axiomatic treatment.  
However, we develop here some
filters on $S^\delta_\kappa(A) = \{\langle A_i:i < \delta \rangle:A_i$ a
subset of $A$ of power $\le \kappa$, $A_i$ increasing continuous$\}$, using
those we prove $(*)$.

Further work on the situation in regular cardinals is in \cite{Sh:161}, and
the author eliminates $Ax I^*,V$ from the axiomatic system (see \cite{Sh:266}, Hodges \cite{Ho81}).
\bigskip

\begin{theorem} \label{b52.1}
For a $\lambda$ singular, $\chi^* < \lambda$ assume
$\bfF$ is a set of pairs $(A,B),(\bfF$ for free) 
$A,B \subseteq \lambda$ satisfying the axioms (II-IV,VI,VII) below.  
Let $A^*,B^* \subseteq \lambda$, then $B^*/A^* \in \bfF$ 
if $B^*/A^*$ is $\lambda$-free in a weak sense
which means (see Definition below):
\mn
\begin{enumerate}
\item[$(*)_0$]   for the $D_{\chi^*}(B^*)$-majority of $B \in [B^*]
^{< \lambda}$ we have $B/A^* \in \bfF$ or just
\sn
\item[$(*)_1$]  for a club of $\mu < \lambda$ we have 
$$\big\{\mu < \lambda : \{B \in [B^*]^\mu:B/A^* \in \bfF\} \in E^{\mu^+}_\mu(B^*)\big\}$$ 
contains a club of $\lambda$ all $B \in [B^*]^\mu$ (see below), $B/A^* \in
\bfF$
\end{enumerate}
\mn
or at least
\mn
\begin{enumerate}
\item[$(*)_2$]  for some set $C$ of cardinals $< \lambda$, unbounded in
$\lambda$ and closed (meaningful only if cf$(\lambda) > \aleph_0$) for
a $E^{\mu^+}_\mu(B^*)$-positive set of $B \in [B^*]^\mu$, $B/A^* \in \bfF$.
\end{enumerate}
\end{theorem}
\bigskip

\noindent
The axioms are:

\noindent
\underline{Ax II}:  
$B/A \in \bfF \Leftrightarrow A \cup B/A \in \bfF$.
\medskip

\noindent
\underline{Ax III}:  if $A \subseteq B \subseteq C$, $B/A \in \bfF$, 
$C/B \in \bfF$ then $C/A \in \bfF$.
\medskip

\noindent
\underline{Ax IV}:  if $A_i$ is increasingly continuous for $i \le \theta =
\cf(\theta),A_{i+1}/A_i \in \bfF$ then $A_\theta/A_0 \in \bfF$.
\medskip

\noindent
\underline{Ax VI}: if $A/B \in \bfF$ then for the $D_{\chi^*}$-majority of
$A' \subseteq A$, $A'/B \in \bfF$.
\medskip

\noindent
\underline{Ax VII}:  if $A/B \in \bfF$ then for the 
$D_{\chi^*}$-majority of $A' \subseteq A$, $A/B \cup A' \in \bfF$.
\bigskip

\begin{definition}\label{b52.2}
1) Let for $D$ a function giving for any set $B^*$
a filter $D(B^*)$ on ${\clP}(B^*)$ (or on $[B^*]^\mu$), then to say for
the $D$-majority of $B \subseteq B^*$ (or $B \in Y^*$) we have $\varphi(B)$
means $\{B \subseteq B^*:\varphi(B)\} \in D(B^*)$ or $\{B:B \in Y^*,\neg
\varphi(B)\} = \varnothing \mod D$. 

\noindent
2)  Let $D_\mu(B^*)$ be $\{Y:Y \subseteq {\clP}(B^*)$ such that for some
algebra $M$ with universe $B^*$ and $\le \mu$ functions, $Y \supseteq \{B
\subseteq B^*:B \ne \varnothing$ closed under the functions of $M\}$.

\begin{equation*}
\begin{array}{clcr}
E^{\mu^+}_\mu(B^*) = \bigl\{ Y \subseteq [B^*]^\mu:&\text{for some }
\chi,\ x \text{ such that } \{B^*,x\} \in {\cH}(\chi)  \\
  &\text{and if } \bar M = \langle M_i:i < \mu^+ \rangle \text{ is an
increasingly} \\
  &\text{continuous sequence of elementary submodels of} \\
  &({\cH}(\chi),\in) \text{ such that } x \in M_0 \text{ and} \\
  &\bar M \rest (i+1) \in M_{i+1} \text{ then for some club } C
\text{ of } \mu^+, \\
  &i \in C \Rightarrow M_i \cap B^* \in Y \bigr\}
\end{array}
\end{equation*}

\mn
(see \cite{Sh:52}, \cite{BD}). Now the proof in \cite{Sh:52} of singular compactness 
used two 
axioms which the author eliminated later; 
see \cite{Sh:266}, 
Hodges \cite{Ho81}. 

There are some cases of incompactness (\cite{Sh:267}, \cite{Sh:347}).
\end{definition}
\bigskip

\noindent
(B53) \quad \underline{Models with few non-isomorphic expansions}, 
IJM 28(1971), 331-338 (with A. Litman).

If $M$ is a countable model 
(with countable language) the number of expansions
of $M$ by one one-place predicate is, up to isomorphism, $\aleph_0$ or
$2^{\aleph_0}$.  Moreover, we characterize when the number is $\aleph_0$:
when $M$ is definable in $N \times \omega$ for some finite model $N$.
\bigskip

\noindent
(B54) \quad \underline{The lazy model theorist's guide to stability,
six days of model theory},\\ Proc. of a conference in Louvain-la-Neuve, 
March 1975, edited by Paul Henramd, Pub. by Paul Castella, Switzerland,
1661 Albeuve (1978), 9-76.

This paper is an introduction to classification theory and an outline of the
main results updated to 1975.  The main goal is to convince the reader that
stability is useful for algebraic theories, and though it was developed for
elementary classes, it can be adapted to others, mainly:
\mn
\begin{itemize}
\item  the existentially closed models of a universal theory with the
joint embedding property and the amalgamation property
\sn
\item  the existentially closed models of a universal theory with the JEP.
\end{itemize}
\mn
We assume knowledge of only basic model theory (as in the book of Chang and
Keisler); stability is defined and some of the harder combinatorial theorems
are just quoted and appear with an outline of proof.

Morley's categoricity theorem is proved, uniqueness of prime model for
stable countable theory, characterization of stability using rank, the basic
properties of non-forking (sometimes with simpler proofs than in the book
\cite{Sh:a}).  We prove that for a ring $R$, if $K$ is the class of 
$R$-modules then either every $R$-module $M$ is the direct sum of modules
of $\le |R| + \aleph_0$, \oor \, in every $\lambda > \aleph_0$ there are
$2^\lambda$ non-isomorphic $R$-modules; the first case occurs iff the
theory is superstable.  Application to group theory (chain condition) and
field are discussed.

It is noted (with Macintyre) that by \cite{Sh:39}, some theories of separably
closed field of characteristic $p$ are stable (ie. the reducts of the theory
of differentially closed fields of characteristic $p$), and by a similar 
proof, all such theories are stable.

Since this paper is very condensed and is itself a summary of results, we
could not mention in this summary all the theorems.  We gave just a 
representative sample list.
\bigskip

\noindent
(B55) \quad \underline{Universal locally finite groups}, 
J. of Alg. 43(1976), 168-175 (with Macintyre).

Call $G$ a u.f.l. (universal locally finite) group if:
\mn
\begin{enumerate}
\item[(a)]   every finitely generated subgroup is finite
\sn
\item[(b)]  every finite group can be embedded into $G$
\sn
\item[(c)]  every isomorphism by one finitely generated subgroup to
another is induced by an inner automorphism.
\end{enumerate}
\mn
We prove that for $\lambda > \aleph_0$, there are $2^\lambda$ u.l.f. 
non-isomorphic groups, no one embeddable into another (using \cite{Sh:16}).
\bigskip

\noindent
(B56) \quad \underline{Refuting Ehrenfeucht Conjecture on Rigid models}, 
IJM 25(1976), 273-288.

We prove that the class of cardinalities in which a first-order sentence has
a rigid model can be very complicated.  In fact, we essentially characterized
those classes: if GCH holds any $\Sigma^1_2$ class (of pure second order
logic) occurs.
\bigskip

\noindent
(B57) \quad \underline{The complete finitely axiomatized theories 
of order are dense}, IJM 23(1976), 200-208 (with A. Amit).

We prove that every sentence of the theory of linear order which has a model,
has a model with finitely axiomatized theory.
\bigskip

\noindent
(B58) \quad \underline{Decidability of a portion of the 
predicate calculus}, IJM 28(1977).

During this century many fragments of first-order logic were proved to be
decidable or undecidable.  The most natural fragments are classes of prenex
sentences defined by restrictions on prefix and numbers of monadic, dyadic,
etc., predicate and function symbols.  Such classes are called standard
below.  The decision problem for standard classes without equality 
is settled.
See the complete picture in the article, ``The decision problem for standard
classes" by Y. Gurevich in JSL 41(1976).  According to that article the only
open case of the decision problem for standard classes with equality and at
least one function symbol is the class (let's call it $K$) of
$\exists \ldots \exists \forall \exists \ldots \exists$ sentences with
arbitrary predicate symbols and exactly one function symbol, which is unary.
Here we settle the decision problem for standard classes with equality and
at least one function symbol by proving decidability.  Also the decision
problem for satisfiability in finite models is decidable.  Let us say a few
words about the method.  Given $\varphi \in K$ we study closure properties
of the class of models of $\varphi$ mainly under what operations it is
preserved.

See the book by Berger Gurevich.
\bigskip

\noindent
(B59) \quad \underline{Singular Cohomology in $\bfL$}, IJM 26(1977), 313-319
(with H.L. Hiller). 
\begin{theorem}
$(\bfV = \bfL)$.  If $G$ is a $\kappa$-free non-free abelian group
of power $\kappa$ then $\Ext(G,\bbZ)$ has power $\kappa^+$ hence for no
torsion free $G,\Ext(G,\bbZ)$ has power $\aleph_0$.
\end{theorem}

A consequence is 
that there does not exist a topological space $X$ and integer
$n$ such that $H^n(X,\bbZ) = \bbQ$ ($H^n$ -- ordinary, singular cohomology).

In \cite{Sh:125} we construct a universe of set theory where there exists a
topological space $X$ and an integer such that $H^n(X,\bbZ) = \bbQ$.
\bigskip

\noindent
(B60) \quad \underline{Possible orderings of an indiscernible sequence}, 
 Bull. London, Math. Soc. 9(1977), 
212-215 (with W. Hodges and A. Lachlan).

Let $n$ be a positive integer and $X$ a set of cardinality $\ge \max
(2n-3)$, $n+2$.  Suppose $(X,<)$ is indiscernible with respect to an
$n$-relation $R$, i.e. $<$ linearly orders $X$, and whether $Rx_1,\dotsc,x_n$
holds depends only on the relative order of $x_1,\dotsc,x_n$ in $X$.  

Then one of the following holds:
\mn
\begin{itemize}
\item  for all linear orderings $<_1$ of $X,(X,<_1)$ is indiscernible w.r.t.
$R$;
\sn
\item  if $<_1$ linearly orders $R$ than $(X,<_1)$ is indiscernible w.r.t.
$R$ iff $<_1$ comes from $<$ by moving an initial segment to the end and
maybe reversing the order;
\sn
\item  if $<_1$ linearly orderings $X$ then $(X,<_1)$ is indiscernible w.r.t.
$R$ iff $<_1$ is either $<$ or $>$, except that at each end of $(X,<)$ some
elements may be permuted (but at most $n$ in all).  The theorem is equivalent
to an earlier result of Fransnay not mentioning indiscernibles, but
our bounds improve his.
\end{itemize}
\bigskip

\noindent
(B61) \quad \underline{Interpreting set theory in the 
endomorphism semi-group of a free algebra} \underline{or in a category}, 
Ann. Sci. Univ. Clermont Sec. Math. fase 13(1976), 1-29.

In the category of any variety we interpret set theory, essentially.  
Moreover, in the semi-group of endomorphism of $F_\lambda$, 
the free algebra with $\lambda$ generators, we interpret 
$(\cH(\lambda^+),\in)$, essentially.  
The essentially indicate that we assume some terms call beautiful does
not exist, but even if they exist we get the best possible results.

The method is using stationary sets and deepness of trees.
\bigskip

\noindent
A related problem is about the group of automorphism of $F$, $\Aut(F_\lambda)$,
and we believe 
\begin{conjecture}\label{b61}
For varieties satisfying $(*)$ below, a similar answer
holds, whereas if $(*)$ fails the group is like the group of permutations
of $\lambda$ (see [24,25]).
\mn
\begin{enumerate}
\item[$(*)$]   Let $F$ be freely generated by $\{x_i:i < \omega + \omega\}$.  
Then there is an automorphism $h$ of $F$, $h(x_i)=x_i$ for $i <\omega$ 
and $h(x_\omega) \notin \langle x_i:\omega \le i < \omega + \omega \rangle_F$.
\end{enumerate}
\end{conjecture}
\bigskip

\noindent
(B62) \quad \underline{The theorems of Beth and Craig in Abstract 
Model theory I, The abstract} \underline{setting}, 
Trans AMS 256(1979) (with J.A. Makowsky).

In the context of abstract model theory various definability properties,
their interrelations and their relation to compactness are investigated.

Beth's theorem together with a Feferman-Vaught theorem for tree-like sums
implies a weak form of Robinson's consistency lemma, hence e.g. assuming
compactness, Craig theorem follows.  So in many cases Craig and Beth's
theorems are equivalent.  Also, essentially, Robinson consistency lemma
together with the Feferman-Vaught theorem for pairs implies full compactness.
\bigskip

\noindent
(B63) \quad \underline{The Hanf number of the first-order 
theory of Banach spaces}, Trans AMS 244(1978) (with J. Stern). 

In this paper, we discuss the possibility of developing a nice, i.e. first
orer theory for Banach spaces: the restrictions on the set of sentences for
recent compactness arguments applied to Banach spaces as well as for other
model-theoretic results are both natural and necessary; without them we rove
here that we essentially get a second order logic with quantification over
countable sets.  Especially, the Hanf number for sets of sentences of the
first-order theory of Banach spaces is exactly the Hanf number for the second
order logic of binary relations (with the second order quantifiers ranging
over countable sets).
\bigskip

\noindent
(B64) \quad \underline{Whitehead group may not be free even 
assuming Ch.I}, IJM 29(1978), 239-247.

In spite of its name this is a paper in set theory.  Its purpose is to prove
that many consequences of MA are consistent with GCH.  More specifically,
it is possible that GCH holds, $S \subseteq \omega_1$ is a stationary,
co-stationary set of limit ordinals, 
$\omega_1 \setminus S$ is not small in fact
$\diamondsuit^*_{\omega_1 \setminus S}$ holds, but e.g. every ladder system
$\bar \eta = \langle \eta_\delta:\delta \in S \rangle$ can be
$\aleph_0$-uniformized where ``$\bar \eta = 
\langle \eta_\delta:\delta \in S \rangle$ is a ladder system" means 
each $\eta_\delta$ is an increasing
$\omega$-sequence converging to $\delta$ and $\kappa$-uniformized means that
for every $\bar c = \langle c_\delta : \delta \in S \rangle$, 
$c_\delta \in {}^\omega \kappa$ there is $h : \omega_i \rightarrow \kappa$ 
such that for every $\delta \in S$, for large enough $n$, 
$c(\eta_\delta(n)) = c_\delta(n)$.
By \cite{Sh:65}, $\omega_1 \setminus S$ has to be stationary.
So we prove that $\diamondsuit_{\omega_1} \ne \diamondsuit_S$ 
($S \subseteq \omega_1$ stationary).

We use straightforwardly countable support iteration of length $\omega_2$;
the main problem is why $S$ remains stationary.  For this we take
$N \prec ({\cH}(\aleph_3),\in)$ and build a sequence of $\bar q^n = 
\langle q_\eta: \eta\in T_n \rangle$, $T_n$ a finite tree, and in the 
end get by inverse limit, a tree of candidates for being a condition 
generic over $N$.  Using $\omega_1 \setminus S$ we then prove at 
least one is contained in a condition.  
\bigskip

\noindent
(B65) \quad \underline{A weak form of the diamond which follows 
from $2^{\aleph_0} < 2^{\aleph_1}$} IJM 29(1978), 239-247 (with K. Devlin).

We prove that weak CH (i.e. $2^{\aleph_0} < 2^{\aleph_1}$) implies a weak
variant of the diamond, explicitly:
\mn
\begin{enumerate}
\item[$(\Phi)$]  for each function $F:{}^{\omega_1 >}2 \rightarrow 2$
there is $g \in {}^{\omega_1}2$, such that for any $f \in {}^{\omega_1}2$
\[
\{\alpha < \omega_1 : F(f \rest \alpha) = g(\alpha)\} \text{ is stationary}.
\]
\end{enumerate}
\mn
Note that the ``guessing power" is being able to say ``yes or no" (i.e. 0 or
1), i.e. the range of $g$ is 2 (not even 3).  Also the guesses are on 
$\omega_1$, not any stationary subset $S$, so we can phrase $\Phi(S)$, and
let us call $S$ small if it fails.  The small sets form a normal idea.  There
are generalization for higher cardinals.

A corollary is that if $\Gamma(G) \subseteq \omega_1$ (see \cite{Sh:44}) is
not small then $G$ is not Whitehead.  Also if $S \subseteq \omega_1$ is not
small, $\bar \eta = \langle \eta_\delta : \delta \in S \rangle$ a ladder
system then it is not even 2-uniformizable (see \cite {Sh:64}).
\bigskip

\noindent
(B66) \quad \underline{End extensions and numbers of countable models}, 
JSL 43(1978), 550-562.

We prove that:
\mn
\begin{itemize}
\item  every model of $T = \mathrm{Th}(\omega,<,\ldots)$ for $T$ countable has a
(proper) end extension (so addition and multiplication are not used, 
hence definition by induction is not available);
\sn
\item  every countable theory with an infinite order and Skolem functions
has $2^{\aleph_0}$ non-isomorphic countable models; (instead of 
Skolem functions, existence of arbitrarily large finite intervals 
can be assumed);
\sn
\item  if every model of $T$ has an end extension, then every $|T|$-universal
model of $T$ has an end extension definable with parameters.
\end{itemize}
\bigskip

\noindent
(B67) \quad \underline{On the number of minimal models}, JSL 43(1978), 
475-480.

We prove that for every $\kappa$, $1 \le \kappa \le \aleph_0$, there is a
(countable) complete theory $T$, with no prime model, and exactly $\kappa$
minimal models (up to isomorphism).  The examples are specifically built for
this purpose.

It is clear that if the number is $> \aleph_1$ it is $2^{\aleph_0}$.
\bigskip

\noindent
(B68) \quad \underline{J\'onsson algebras in successor cardinals}, 
IJM 30(1978), 57-64.

We shall show here that in many successor cardinals $\lambda$, there is a
J\'onsson algebra (in other words $\mathrm{Jn}(\lambda)$ or $\lambda$ is not 
a J\'onsson cardinal).  In connection with this we show that, e.g., for every 
ultrafilter $D$ over $\omega$, in $(\omega_\omega,<)^\omega/D$ there is no 
increasing sequence of length $\aleph_{(2^{\aleph_0})^+}$ and a lemma on the 
existence of l.u.b. We prove e.g., $\mathrm{Jn}(\aleph_{\omega +1})$ when 
$2^{\aleph_0} \le \aleph_{\omega + 1}$ and $\mathrm{Jn}(2^{\aleph_0})$ when 
$2^{\aleph_0} = \aleph_{\alpha +1}$, 
$\alpha < \omega_1$ and similar results for higher cardinals.  
\bigskip

\noindent
(B69) \quad \underline{On a problem of Kurosh, J\'onsson groups and 
applications work problem}, II, edited by Adjan, Boone and Higman 
North Holland Publ. Col., 1979

We prove some results in group theory in model theoretic spirit. 

\noindent
1) Construct J\'onsson groups of cardinality $\aleph_1$, and others as well
($G$ is J\'onsson group of cardinality $\lambda$ ($> \aleph_0$) if it has
no proper subgroup of cardinality $\lambda$); this solves a problem 
of Kurosh. 

\noindent
2) Our group is simple with no maximal subgroup; so it follows that taking
Frathini subgroups does not commute with direct products. 

\noindent
3) Assuming $2^{\aleph_0} = \aleph_1$, our group has no non-trivial Hausdorff
topology; this answers a question of Markov.

We use small cancellation theory, but quote exactly what we use; we assume
only knowledge of naive set theory and little group theory (except the
theorem we quote).
\bigskip

\noindent
(B70) \quad \underline{Modest theory of short chains}, II, 
JSL 44(1979), 491-502 (with Yuri Gurevich).

A chain (i.e. a linearly ordered set) is \emph{short} if it embeds neither
$\omega_1$ nor $\omega^*_1$.  Shortness is easily expressible in the monadic
language of order.  Given a chain $C$ and a positive integer $p$ we define
(in the monadic language of order) $p$-\emph{modest} subsets 
in $C$.  A subset of $C$ is \emph{modest} in $C$ if it 
is $p$-modest in $C$ for every $p$. A chain is \emph{modest} if it is modest 
in itself, it is \emph{absolutely modest} if it is modest in its Dedekind completion.
\bigskip

\begin{theorem}
\label{b70.1}  The monadic theory of modest short chains
coincides with that of countable chains.   
\end{theorem}

\noindent
Decidability of the latter theory is due to Rabin; it is reproved here. 
\bigskip

\begin{theorem}
\label{b70.2}
The monadic theory of any non-modest short chain
is undecidable under the Continuum Hypothesis.  (This use of the Continuum
Hypothesis is eliminated in a later paper ``Monadic theory of order and
topology in ZFC" by the same authors.)
\end{theorem}
\smallskip

\noindent
The \emph{modest theory} of a chain $C$ is the 
theory of $C$ in the monadic
language of orders when the set variables range over absolutely modest 
subsets of $C$.
\bigskip

\begin{theorem}
\label{b70.3}
The modest theory of the real line $R$ is
decidable; it coincides with the theory of $R$ with quantification over
countable sets, with the theory of $R$ with quantification over sets of
cardinality $< 2^{\aleph_0}$, and with the modest theory of any
complete short chain without jumps and ends having an everywhere 
dense absolutely modest subset.  
\end{theorem}
It is proved also that Theorem 3 is in some sense best possible.
\bigskip

\noindent
(B71) \quad \underline{A note on cardinal exponentiation}, 
JSL 45(1980), 56-66.

For an $\aleph_1$-complete filter $D$ on $I$, for $f : I \rightarrow \Ord$, 
let: $\|f\|_D$ (the rank of $f$) be $\sup\{\|g\|_D:g <_D f\}$ and
$$T_D(f) = \sup\big\{|G|:g \subseteq I(\Ord),\ (\forall g \in G) g <_D f,\ 
(\forall g_1 \ne g_2 \in G)[g_1 \ne_D g_2]\big\}.$$

Our results are:
\mn
\begin{itemize}
\item  improvements of the bounds on $\|f\|_D$, e.g. if 
$|\alpha| = \beth_\omega$ and $|I| < \beth_\omega$ then $\| \alpha\|_D <
\beth^+_\omega$, hence
\sn
\item  corresponding bound on $2^{\aleph_\alpha},\aleph_\alpha$ strong 
limit, $\cf(\alpha) \ge \aleph_1$
\sn
\item  bounded on $2^{\aleph_\alpha}$ in new cases: if 
$\aleph_\lambda(\aleph_0)$ is the first $\aleph_\alpha 
= \alpha$ of cofinality $\lambda$ and
$\aleph_{\omega_1}(\aleph_0)$ is strong limit then 
$2^{\aleph_{\omega_1}(\aleph_0)} < \aleph_{\omega_2}(\aleph_0)$ 
\underline{provided} that the Chang conjecture holds.
\end{itemize}
\mn
The method is that in the proof rather than fixing a filter (say the club
filter) we look at all possible filters, say normal filters on $\omega_1$ 
(doing induction together) and that we get bounds on the rank by cases of
$T_D(-)$ (a kind of converse to previous results).
\bigskip

\noindent
(B72) \quad \underline{Models with second order properties I. 
Boolean Algebras with no definable} \underline{automorphisms}, AML 14(1978), 57-72.

This is (the first) part of a series of papers, in which we try to build
models with various second order properties.

Examples of such properties are:
\mn
\begin{enumerate}
\item[(a)]  let $T$ be a (first order) theory, and we want a model $M$,
such that $M \models T$ and every Boolean algebra definable in $M$, all its
automorphisms are definable in $M$
\sn
\item[(b)]  we extend first order logic by allowing quantification over 
automorphism of Boolean algebras, and we want to show this logic is compact 
(see \cite[Th. 2.6, pg. 353]{Sh:43}, \cite[Th. 25]{Sh:18}) for a definition 
of a second order quantification, and this amounts to assuming $T$ satisfies 
various schemes, e.g., $T$ has a model expanding $({\cH}(\kappa^+),\in)$
\sn
\item[(c)]   we have a model $M$ of $T$ in which a (definable) Boolean
algebra is rigid, and we want to build other such models.
\end{enumerate}
\mn
We want more concretely to build those models in specific cardinals, and
replace ``automorphisms of Boolean algebras" by automorphisms of other
structures (e.g. ordered fields) branches of trees, etc.

Note that (a) is stronger than (b) which is stronger than (c); and in (b),
(c) we can expand the language, so \wilog \, the theory has Skolem 
functions, 
and has enough build in set theory. But in (a) this is forbidden.  Note that
if we e.g. want to get a $(\lambda^+,\lambda)$-model in which a tree of
height $\lambda^+$ has only definable branches, Chang's original proof of his
two-cardinal theorem is not appropriate, as he extends the language (so to
encode finite sets).  Note also that (b) is good for giving examples of
compact logic.  For example, allowing quantification over automorphisms of
Boolean algebras, gives us a logic stronger than first order even on finite
models (we can say a Boolean algebra is atomic, and has an automorphism of
order two, which moves every atom.  This distinguishes among the finite
Boolean algebras between those with an even number of atoms.  We shall
prove this logic is compact (assuming GCH).

Let us return to this paper.

In \cite{Sh:84} we prove (assuming $\diamondsuit_{\aleph_1}$) that for a
Boolean algebra $B$, the (two-sorted) model $(B,\Aut(B))$ (with the
Boolean operations, and the operations of applying an automorphism (where
Aut$(B)$ is the group of automorphisms of $(B)$) is elementarily equivalent
to $(B',\text{Aut}(B'))$ for some $B'$ of cardinality $\aleph_1$.  We can
add more relations and make other strengthening.  Our first aim was to
generalize this to uncountable cardinalities.  This is done in \S1, for an
appropriate $\lambda$ our results are stronger; because we prove a theory
$T$ has a model in $\lambda^+$ in which every automorphism of $(P(M) \cup
Q(M),R^M)$ is inner (i.e. definable is the model) and we do not expand the
language to have, essentially, some replacement axioms.  Here the
requirements on $\lambda$ are severe; but this paper was written in order to
exemplify the technique developed for solving the problem.

In \S2 we use our technique to prove (a slightly stronger version of) the Chang
two-cardinal theorem without expanding the language.

In some sense our problem is to omit types of power $\lambda$ in models of
power $\lambda^+$.
\bigskip

\noindent
(B73) \quad \underline{Models with second order properties. II. 
Trees with no undefined branches}, AML 14(1978), 73-87.

We prove several theorems of the form: a first order theory $T$ has a model
$M$ (sometimes with additional conditions) such that (some) trees defined
in $M$, have no branches except those defined in it.  For this we omit types
of cardinality $\lambda$ for constructing models in $\lambda^+$ using
``strong splitting".  We have some applications, e.g. an example for compact 
logic $\bbL(Q)$, where in $\bbL_{\omega_1,\omega}(Q)$ well-ordering is 
definable, its compactness is proved in ZFC, and it is stronger than first 
order logic even for countable models; and eliminate extra set-theoretic axioms 
(like $\diamondsuit_{\aleph_1}$) from some theorems.  The paper also deals with 
constructing models of cardinality $\aleph_1$.

Also we prove the absoluteness of the answer to ``does 
$\psi \in \bbL_{\omega_1,\omega}(\exists^{\ge \aleph_1})$ have a model $M$ such that every
tree with sets of levels of cofinality $\aleph_1$ has no (full) branch
which is not (first order) not definable (with parameters)?".  So if we can
find such a model, assuming $\diamondsuit_{\aleph_1}$, then we can really
construct one\footnote{if we use a countable fragment ${\cL}$ of
$\bbL_{\omega_1,\omega}(\exists^{\ge \aleph_1})$ to which $\psi$ belongs and
consider ${\cL}$-definability with parameters, straightforwardly there is
such a model if $\diamondsuit_{\aleph_1}$ hence provably in ZFC} (e.g. a
model of ZF or PA: the ``classes" of it are branches of such a tree).
\bigskip

\noindent
(B74) \quad \underline{Appendix, Vaught two cardinal theorem revisited}.

We gave an alternative proof to Vaught's theorem $(\lambda^+,\lambda)
\rightarrow (\aleph_1,\aleph_0)$ which reveal the corresponding partition
theorems from \cite{Sh:8}, i.e. we characterize them.
\bigskip

\noindent
(B75) \quad \underline{A Banach space with few operators}, 
IJM 30(1978), 181-191.

Assuming the axiom (of set theory) $\bfV = \bfL$ (explained below), we construct a
Banach space with density character $\aleph_1$ such that every (linear
bounded) operator $T$ from $B$ to $B$ has the form $aI + T_1$, where $I$ is
the identity and $T_1$ has a separable range.  The axiom $\bfV = \bfL$ means that
all the sets in the universe are in the class $\bfL$ of sets constructible from
ordinals; in a sense this is the minimal universe.  In fact, we make use
of just one consequence of this axiom, $\diamondsuit_{\aleph_1}$ proved by
Jensen, which is widely used by mathematical logicians.
\bigskip

\noindent
(B76) \quad \underline{Independence of strong partition relations 
for small cardinals and the free} \underline{subset problem}, J.45(1980), 505-509.

We prove the consistency of ZFC + GCH with the following (assuming 
$\mathrm{con}(\mathrm{ZFC})$ + ``there are $\omega$ measurable cardinals''):
\mn
\begin{enumerate}
\item[$(*)$]  for any function $f$ from finite subsets of $\aleph_\omega$
to $\omega$ there are pairwise disjoint $S_n \subseteq \omega_{2n}$
$(1 \le n < \omega)$ satisfying $|S_n| = \aleph_{2n}$ such that: if
$u,v$ are finite subsets of $\aleph_\omega$ and $(\forall n)[|S_n \cap u| =
|S_n \cap v| \le 1]$ \then \, $f(u) = f(v)$.
\end{enumerate}
\mn
On the free subset problem we prove the following: $(*)$ implies that any
algebra with countably many functions and power $\aleph_\omega$ has a free
subset $A$ of power $\aleph_0$, i.e. $(\forall a \in A)$ ($a$ is not in the
subalgebra generated by $A \setminus \{a\}$).

We get similar independence results for $\aleph_\alpha > |\alpha|$, i.e.
a free subset of power $|\alpha|$, and prove we cannot get a free set of
power $|\alpha|^+$ in $\aleph_\alpha$.  Continued in \cite{Sh:124}.
\bigskip

\noindent
(B77) \quad \underline{Existentially-closed groups in $\aleph_1$ 
with special properties}, Bull. Greek Math. Soc. 18(1977), 17-27.

We prove: for any countable e.c. group (= existentially closed group) $G^*$,
there is an e.c. group $M \equiv {}_{\infty,\omega}G^*$ of cardinality
$\aleph_1$ such that:
\mn
\begin{enumerate}
\item[(i)]   $M$ is complete; i.e. every automorphism is inner
\sn
\item[(ii)]  (CH) $M$ has no uncountable abelian subgroup.
\end{enumerate}
\bigskip

\noindent
(B78) \quad \underline{Hanf number of omitting types for 
simple first-order theories}, JSL 44(1979), 319-324.

Let $T$ be a complete countable first-order theory such that every
ultrapower of a model of $T$ is saturated.  If $T$ has a model omitting a
type $p$ in every cardinality $< \beth_\omega$, then $T$ has a model omitting
$p$ in every cardinality.  There is also a related theorem, and an example
showing the $\beth_\omega$ cannot be improved.  Continued in [Sh:334].
\bigskip

\noindent
(B79) \quad \underline{On uniqueness of prime models}, JSL 44(1979), 215-220.

We prove there are theories (stable or countable) for which over every $A$
there is a prime model but it is not necessarily unique.  We also give a
simplified proof of the uniqueness theorem for countable stable theories.
\bigskip

\noindent
(B80) \quad \underline{A weak generalization of MA to higher cardinals}, 
IJM 30(1978), 297-306.

We generalize MA to $\aleph_1$-complete forcing notions satisfying a
strengthening of the $\aleph_2$-c.c. described in $(*)$ below, and prove its
consistency with ZFC + CH.  (The number of dense sets is $< 2^{\aleph_1}$,
of course).  Where
\mn
\begin{enumerate}
\item[$(*)$]   for any $p_i \in P$ (for $i < \aleph_2$) there are
regressive functions $f_n : \omega_2 \rightarrow \omega_2$, such that if
$\alpha,\beta < \aleph_1$, $\bigwedge\limits_{n < \omega} f_n(\alpha) =
f_n(\beta)$ \then \, $p_\alpha,p_\beta$ have a least upper bound.
\end{enumerate}
\mn
We get similar consistency results (with $ZFC + CH + ``2^{\aleph_1} =
2^{\aleph_2}$") like a uniformization property for suitable $\langle 
\eta_\delta:\delta \in S^2_1 \rangle$.

As a result we get the consistency of the decidability of the monadic theory
of $\omega_2$.
\bigskip

\noindent
(B81) \quad \underline{The consistency with CH of some consequences 
of Martin Axiom plus}\\ \underline{$2^{\aleph_0} < 2^{\aleph_1}$}, 
 31(1978), 19-33 (with U. Abraham and K. Devlin).

The main conclusion is the consistency with GCH of:
\mn
\begin{enumerate}
\item[$(*)(1)$]  for any ladder system $\langle \eta_\delta : \delta < \omega_1 \rangle$ 
($\delta$ limit, $\eta_\delta$ an $\omega$-sequence converging to $\delta$) 
and $c_\delta \in {}^\omega\omega$, for some
$h : \omega_1 \rightarrow \omega$ we have $(\forall \delta)(\exists^\infty k)
[h(\eta_\delta(k)) = c_\delta(k)]$
\sn
\item[$(*)(2)$]  if $G$ is a graph, $\omega_1$ its set of models, 
$$(\forall \alpha,\beta) \big[ \alpha + \omega \le \beta \rightarrow
(\exists^{< \aleph_0} \gamma < \alpha)\langle \gamma,\beta \rangle \in G \big]$$
then $G$ has chromatic number $\le \aleph_0$.
\end{enumerate}
\mn
The method the known iteration of forcing by Suslin trees following a work
of Jensen; and a general principle SAD implying $(*)$ is phrased, and
its consistency proved.
\bigskip

\noindent
(B82) \quad \underline{Models with second order properties, 
III: omitting types for $\bbL(Q)$}, Archive f. Math Logic, 21(1981), 1-11.

A known technique for building model a power $\aleph_1$ is by building an
increasing continuous chain $M_\alpha$ $(\alpha < \omega_1)$ of countable 
models, and in stage $\alpha$ we are allowing to promise to omit 
countably many types with parameters from $M_\alpha$ ``strongly omitted" by
$M_\alpha$.  We prove the parallel theorem with $\aleph_0,\aleph_1$
replaced by $\lambda,\lambda^+$ provided that $\diamondsuit_\lambda$ holds
(or $\lambda$ is strongly inaccessible).  However, we prove 
$\diamondsuit_\lambda$ holds for any successor $\lambda \ne \aleph_1$ if GCH
holds. 

The obstacle for straightforward generalization is preserving
the induction hypothesis in limit stages of cofinality $< \lambda$.

The main idea is to use $\diamondsuit_\lambda$, while constructing
$M_{\alpha +1}$ by approximation of power $< \lambda$, to ``guess" obstacles
in the future.
\bigskip

\noindent
(B83) \quad \underline{Existentially closed structures in 
the power of the continuum}, Ann. Pure Appl. Logic 26(1984), 123-148 (with Giorgetta)

We deal with three cases:
\mn
\begin{itemize}
\item  $G$ is a countable universal locally finite group
\sn
\item  $G$ is a countable existentially closed group,
\sn
\item  $G$ is a countable existentially closed division ring.
\end{itemize}
\mn
In each case $K = \{H : H$ is $\bbL_{\infty,\omega}$-equivalent to $G\}$ which
here means: every embedding of a finitely generated subalgebra of $G$ into
$H$ can be extended to an embedding of $G$ into $H$, and every countable
subalgebra of $H$ can be embedded into $G$.

We prove that $K$ has members of power $2^{\aleph_0}$; in fact, many quite
complicated members which e.g. do not have $\aleph_1$ pairwise commuting
elements.  The method is general: we build the model of power $2^{\aleph_0}$
as, essentially, an inverse limit of \\
$\langle M^n_u:u \in \clP(n) \rangle$, $M^n_u$ finitely generated, 
so the problems become finitary: for $N \subseteq N_0$, $N_1$ 
finitely generated, amalgamate $N_0,N_1$ over $N$ with some extra
requirements.  For those we have some criteria, but the main 
work is specific to each case.
\bigskip

\noindent
(B84) \quad \underline{On elementary equivalence of 
automorphism groups of Boolean Algebras,} \underline{downward Skolem-L\"owenheim 
theorems and compactness of related quantifiers},\\ JSL 45(1980), 265-283 (with M. Rubin).
\bigskip

\noindent
(B85) \quad \underline{A note on the normal Moore space conjecture} 
Can. J. Math., Vol.XXXI, 2(1979), 241-251 (with Devlin).

In this paper we show that if CH be assumed, then Jones' spaces are not 
normal and that the GCH does not lead to a positive solution to the Jones
conjecture.  A brief survey of the progress on the problem to date is also
included.  Jones' spaces are closely related to Aronszayn trees.
\bigskip

\noindent
(B86) \quad \underline{Souslin properties and Tree Topologies} 
 Bull. London Math Soc., (3), 39(1979), 537-552.
\bigskip

\noindent
(B87) \quad \underline{Classification theory for non-elementary classes. I.
The number of uncount-} \underline{able models of $\psi \in \bbL_{\omega_1, \omega}$}, 
Israel J. Math., 46(1983), 212-273.
\bigskip

\begin{tmt}\label{b87}
Assuming $\forall n < \omega$, $2^{\aleph_n} < 2^{\aleph_{n+1}}$; 
if $\psi \in \bbL_{\omega_1,\omega}$ ($\tau_\psi$ countable) has an
uncountable model then:
\mn
\begin{itemize}
\item  there exists $n$ positive such that $I(\aleph_n,\psi) = 2^{\aleph_n}$
\oor
\sn
\item  $\psi$ has models in every uncountable cardinality, and \Iff \,
$(\exists \lambda > \aleph_0)I(\lambda,\psi) = 1$ \then \, $(\forall \lambda
> \aleph_0)I(\lambda,\psi) = 1$.
\end{itemize}
\end{tmt}
\bigskip

As in \cite{Sh:48} we consider a class of atomic models $K$ of some first
order complete theory.

It is proved that it is enough to prove the Main Theorem for $K$.  Assuming
$2^{\aleph_0} < 2^{\aleph_1}$ and $I(\aleph_1,\psi) < 2^{\aleph_1}$ basic
stability machinery is introduced, i.e. a substitute to non-forking; and a
substitute to saturated models in elementary classes.

The key to classify the complexity of the class $K$ instead by stability,
to do it by generalized amalgamation properties (remember the role of
$\aleph_0$ - A.P. in \cite{Sh:48}).  The problem to find A.P.-s such that
the failure of those A.P.-s will imply existence of many
non-isomorphic models 
and if those generalized A.P.-s holds then to conclude existence of models
in higher cardinalities.

The paper logically is divided to three parts: first developing basic 
stability theory.  Second, introducing various amalgamation properties 
 proving some interrelations among them and proving that 
if all the A.P.-s holds them
$K$ has models in every cardinality and the Morley categoricity theorem for
$K$.  The third part is by combinatorial work to prove that the failure of
A.P. implies the existence of many non-isomorphic models.

We try to minimalize the use of set theoretical assumptions: the main result
of \cite{Sh:48} is proved from $2^{\aleph_0} < 2^{\aleph_1}$ only.  
Most parts 
of the last part of the paper can be proved from the assumption $(\forall n
< \omega)$ that $2^{\aleph_n} < 2^{\aleph n+1}$.

It is an open problem to prove the Main Theorem in ZFC only, see also
\cite{Sh:88}.  Continued in \cite{Sh:600}.
\bigskip

\noindent
(B88)\quad \underline{Classification theory for non-elementary classes. II.
Abstract elementary}\\ \underline{classes}, Lecture Notes in Math. 1292(1987) 419-497. 

We deal with a class of models $K$ all of the same similarity type, and a
partial order $<$ on $K$ (which stands e.g. for elementary submodels) the only
demand except the obvious is existence of $\lambda(K)$ which is a downward
Skolem-Lowenheim number, and $(K,<)$ satisfies Tarski-Vaught elementary chain
theorems.  It is proved from ZFC only that $I(\aleph_1,K) = 1 \Rightarrow
I(\aleph_2,K) \ge 1$ (this improves \cite{Sh:48}, \cite{Sh:87a}, \cite{Sh:87b}) and assuming
$2^{\aleph_0} < 2^{\aleph_1}$ we show that $1 \le I(\aleph_1,K) 
< 2^{\aleph_1}$ implies $I(\aleph_2,K) \ge 1$, the methods 
are classifying by the amalgamation
properties (as in \cite{Sh:48} and \cite{Sh:87a}, \cite{Sh:87b}).  
Most of the work is preparation to generalize the Main Theorem 
from \cite{Sh:87a}, \cite{Sh:87b} to this context.  
Continued in \cite{Sh:576}, \cite{Sh:600}. 
\bigskip

\noindent
(B89) \quad \underline{Boolean algebras with few endomorphisms}, 
Proc. AMS 14(1979), 135-142.

Any Boolean algebra has some trivial endomorphism - the identity, an
endomorphism into $\{0,1\}$ corresponding to an ultrafilter, and some combinations
of them.  We build (assuming $\diamondsuit_{\aleph_1}$ and then CH only) 
Boolean algebra of power $\aleph_1$, such that every endomorphic is 
definable from finitely many elements and ultrafilters.  
We also generalize a construction of Rubin of
some special Boolean Algebras in $\aleph_1$ to higher cardinals.  Continued
in \cite[Ch.I]{Sh:e}.
\bigskip

\noindent
(B91) \quad \underline{The structure of Ext$(G,\mathbb Z)$ and $\bfV = \bfL$}, 
Math Zeitschrift 162(1978), 39-50 (with H.L. Hiller and M.K. Huber).

Let $G$ be a torsion free abelian group, then it is well know that
$\Ext(G,\bbZ)$ is divisible, hence determine by its $p$-ranks ($p$ zero or
prime) $\nu_p(G)$.

We prove that for non-free $G$ (assuming $\bfV = \bfL$): $\nu_0(G)$ is 
$\min\{2^{|H|}:H$ a subgroup of $G$, $G/H$ free$\}$, hence $\nu_0(G)$ is a
successor cardinal and $\nu_p(G) \le \nu_0(G)$.

We determine the possible $\nu_p(G)$ when $\Hom(G,\bbZ) = 0$, and we
characterize the abelian group to which there exists co-Moore spaces; and
assert that instead $\bfV = \bfL$ it is enough to assume no stationary subset of a
regular cardinal is small.  Continued in \cite{Sh:314}.
\bigskip

\noindent
(B92) \quad \underline{Remarks on Boolean Algebras}, 
Algebra Universalis, 11(198), 77-89.

In \S4, we show that if $B$ satisfies the $\kappa$-chain condition,
$|B| > \lambda = \lambda^{< \kappa}$ then $B$ has large independent subsets,
e.g. of power $\ge \lambda^+$.  We can replace $\lambda^+$ by $\chi = 
\cf(\chi)$ if $(\forall \alpha < \chi)(|\alpha|^{< \kappa} < \lambda)$,
and find an independent subset $A'$ of $A \in [B]^\chi$ of cardinality 
$\chi$.
If $\chi$ is singular, $(\forall \alpha < \chi)(|\alpha|^{< \kappa} < 
\lambda)$ and $A \in [B]^{\chi^+}$ we can find an independent $A' \subseteq
A$ of cardinality $\chi$.

In fact, our results are almost complete.

We prove some theorems on Boolean Algebras.  In \S1, we prove that if $S$ is
a subset of a (Boolean Algebra) $B$, $|S|$ singular strong limit, then 
$S^- = \{a - b:a,b \in S\}$ has a subset of power $|S|$ which is a pie subset
(a subset no two elements of which are comparable).

In \S2 we prove the consistency of $ZFC + 2^{\aleph_0} > \aleph_1$ with
``there is a Boolean Algebra $B$ of cardinality $2^{\aleph_0}$, with no
uncountable pie subset, nor an uncountable chain."

In \S3, we prove that if $|B| = \lambda^+$, $\lambda = \lambda^{< \kappa}$, 
$B$ satisfies the $\kappa$-chain condition, then $B \setminus \{0\}$ is the
union of $\lambda$ ultrafilters (proper, of course).  We shall show in
\cite{Sh:126} that this is not necessarily true when $|B| = \lambda^{++}$, 
and other connected results.
\bigskip

\noindent
(B93) \quad \underline{Simple unstable theories}, Ann. Math. Logic, 19(1980), 177-203.

Being simple is a property of first order complete theories weaker than
stability.  We start their investigation, using a a test question asking if
$|T| < \kappa$, $\mu = \mu^\kappa < \lambda = \lambda^{|T|} \le 2^\mu$, 
$M \models T$, $\|M\| = \lambda$ implies $M$ has a $\kappa$-saturated 
elementary extension of cardinality $\lambda$.  For non-simple the answer 
is no, for simple theories the answer is consistently yes.  We start to 
develop a parallel of \cite[Ch.II,III]{Sh:a}.
\bigskip

\noindent
(B94) \quad \underline{Weakly compact cardinals: a combinatorial proof}, 
JSL 44(1979), 559.

We give here a direct purely combinatorial proof that weak compactness is
equivalent to a combinatorial property (2).  This property (2) is apparently
stronger, and from it, all other usual equivalent definitions and usual
properties of weakly compact cardinals can be deduced.  So this proof may be
useful for books which want to present weakly compact cardinals, but not
logic.
\bigskip

\begin{remark}\label{b94}
For $\mu$ the first inaccessible we can define a 
$\mu$-tree with no $\mu$-branch by:

\begin{equation*}
\begin{array}{clcr}
T = \bigl\{h:&\text{Dom } h \text{ an ordinal } \alpha < \mu,\ h(i) < 1+i, \\ 
  &\text{and for strong limit } i,j \in \text{ Dom } h,\ h(i) \ne h(j)\bigr\}.
\end{array}
\end{equation*}
\end{remark}

\begin{theorem}\label{a94.1}
For $\mu$ strongly inaccessible the following are equivalent:
\mn
\begin{itemize}
\item  $\mu$ is weakly compact, i.e., every $\mu$-tree has a $\mu$-branch
\sn
\item  for every family of functions $f_\alpha:\alpha \rightarrow \alpha \,
(\alpha < \mu)$ there is a function $f:\mu \rightarrow \mu$ such that:
$(\forall \alpha < \mu)(\exists \beta)\alpha \le \beta < \mu$ and $[f_\beta
\restriction \alpha = f \restriction \alpha]$.
\end{itemize}
\end{theorem}
\bigskip

\noindent
(B95) \quad \underline{Canonization theorems and applications}, 
JSL, 45(1981), 345.

We prove canonization theorems generalizing Erd\H{o}s-Rado partition
theorems.  We introduce a very general partition relation which generalizes
$\lambda \rightarrow (\kappa)^\chi_\mu$ in many aspects.  We prove a
corresponding partition theorem (which works simultaneously for sequences of
cardinals).  E.g. assume $\langle \lambda_i:i < \alpha \rangle$ a sequence
of cardinals, $|\alpha| + \kappa < \lambda_i$, 
$\langle \mu_i = \beth_n (\lambda_i)^+ : i < \alpha \rangle$ is strictly 
increasing,  $F : \big[ \bigcup\limits_{i} \mu_i \big]^{< \omega} \rightarrow \theta$, 
then we can find $\langle A_i : i < \alpha \rangle$, 
$A_i \in [\mu_i \setminus \bigcup\limits_{j<i} \mu_j]^{\lambda^+_i}$ and $F$
is restricted to increasing sequences from 
$\bigcup\limits_{i < \alpha} A_i$ does not
depend on the last $n$ arguments.  Moreover, in the inductive proof if we
restrict how the sequences are distributed between the $A_i$'s, e.g.
to those of $\bar n$-kind, where $\bar a$ is of the $\bar n$-kind if
$$n_\ell = \Big| \Big\{k < \ell g(\bar a):\ell g(\bar n) - k+1 = 
\big\{i < \alpha : \varnothing \ne \bar a \cap \mu_i \setminus \bigcup\limits_{j<i} \mu_j \big\}\Big\} \Big|.$$  
If in the last one there is one element (e.g. $n_{\ell g(\bar n)-1}=1$, 
we can gain 2 places provided we increase the
function $\beth_n(\lambda_i)^+$ which mainly consists of $n$ exponentiation,
by some plusses.  If $\alpha < \omega$, we can gain more.  

As an application we complete the answer to the following question: 

Given a Hausdorff space $X$ of cardinality $\lambda$, characterize the pairs of
cardinals $(\lambda,\chi)$ such that $X$ has a subspace $Y$, of cardinality
$\chi$, which is discrete.  Also, generalize Hajnal's free subset theorem for
tuples instead of one-place functions.
\bigskip

\noindent
(B96) \quad \underline{Algebraically closed groups of large cardinality},
JSL 44(1979), 232-552 (with Martin Ziegler).

Let $M$ be a countable algebraically closed group, $\kappa$ an uncountable
cardinal.  We will prove in this paper the following theorems.

\begin{theorem}\label{a96.1}
There is an algebraically closed group $N$ of cardinality
$\kappa$ which is $\bbL_{\infty,\omega}$-equivalent to $M$.
\end{theorem}

\begin{theorem}\label{a96.2}
There is an algebraically closed group $N$ of cardinality
$\kappa$ which is $\bbL_{\infty,\omega}$-equivalent to $M$, and contains a free
abelian group of cardinality $\kappa$.
\end{theorem}

\begin{theorem}\label{a96.3}
There are 
$2^\kappa$ non-isomorphic algebraically closed groups of 
cardinality $\kappa$ which are $\bbL_{\infty,\omega}$-equivalent to $M$.
\end{theorem}

\begin{theorem}\label{a96.4}
There is an algebraically 
closed group $N$ of cardinality
$\kappa$ which is $\bbL_{\infty,\omega}$-equivalent to $M$ and satisfies: every
subgroup of $N$ of uncountable regular cardinality contains a free subgroup
of the same cardinality.
\end{theorem}

The method is to construct things parallel to Ehrenfeucht-Mostowski models,
using things like centralizers of elements.
\bigskip

\noindent
(B97) \quad \underline{Unordered types of ultrafilters}, 
Topology Proc. 3(1979), 199-204 (with M.E. Rudin).

It is proved in ZFC that there are the maximal number of incomparable
ultrafilters on $\omega$, in the Rudin-Kaisler ordering (and similarly for
other cardinals).
\bigskip

\noindent
(B98) \quad \underline{Whitehead groups may not be free even assuming 
CH,II} 
\smallskip

\noindent
(A)  We prove the consistency of $ZFC + 2^{\aleph_0} = 2^{\aleph_1} +$ club
(i.e. $\clubsuit$), where club is the following weakening of diamond:
\mn
\begin{enumerate}
\item[$(*)$]  there is $\langle A_\delta:\delta < \omega_1 \rangle$, 
$A_\delta$ an unbounded subset of $\delta$ such that for any unbounded
$A \subseteq \omega_1$, $\{\delta:A_\delta \subseteq A\}$ is stationary. 
(Under CH, club and diamond were known to be equivalent).
\end{enumerate}
\mn
The method is to start with a model of ZFC + GCH, blow up $2^{\aleph_1}$ and
then collapse $\aleph_1$. 

\noindent
(B)  For simplicity let $\Phi$ be a family of $\omega$-sequences from a set
called $\Dom(\Phi)$; $(\Phi,h)$ has the uniformization property if for any
functions $f_\eta:\Rang(\eta) \rightarrow$ ordinals satisfying
$f_\eta(i) < h(i)$, there is $f:\Dom(\Phi) \rightarrow$ ordinals, such
that for every $\eta \in \Phi$ for all but finitely many $n$, 
$f(\eta(n)) = f_\eta(n)$.

We investigate such properties.
\medskip

\noindent
(B1)  We prove that for a stationary $S \subseteq \omega_1$ and $\eta_\delta$
converging to $\delta$, then property ``$(\langle \eta_\delta:\delta \in S
\rangle,2)$ has the uniformization property" does not depend on $S$ only: in
fact, seemingly trivial changes in $\eta_\delta$ can change its truth value.

\mn
(B2)  We prove that if in the weak diamond (see \cite{Sh:65}), restricting
ourselves to a stationary co-stationary $S$, replacing 2 by 3 give a
non-equivalent statement.

\mn
(B3)  We prove that $\Dom(\Phi)$ may be $\omega,\Phi$ uncountable but still
$(\Phi,2)$ has the uniformization property, i.e. this is consistent with
$ZFC + 2^{\aleph_0} = 2^{\aleph_1}$.   Though this looks like a consequence
of $MA + 2^{\aleph_0} > \aleph_1$, it in fact contradicts MA.
So the result says that there are infinite $A_\alpha \subseteq \omega$ for
$\alpha < \omega_1$ almost disjoint (i.e. $\alpha < \beta \Rightarrow
|A_\alpha \cap A_\beta| < \aleph_0$) and for any $h_\alpha \in
{}^{(A_\alpha)}2$ for $\alpha < \omega_1$ there is $h \in {}^\omega 2$ such
that $\alpha < \omega_1 \Rightarrow h_\alpha \subseteq^* h$.  In fact we
have a condition on $\langle A_\alpha:\alpha < \alpha^* \rangle$ preserved
by finite support iteration such that it guarantees that the natural forcing
of adding $h$ for any given $\bar h = \langle h_\alpha:\alpha < \alpha^*
\rangle$ satisfies it.
\bigskip

\noindent
\underline{Group Theory}:

\noindent
(C) We prove that occurance of non-free Whitehead groups, and $(\Phi,2)$
with the uniformization property $\Phi$ ``non-trivial" are closely related
(for power $\aleph_1$-equivalent).  Hence by (B1), the property of being
Whitehead is quite delicate, maybe depending on seemingly trivial changes.

We prove a set theoretic result which shows that for many non-free $G$,
$\Ext(G,\bbZ)$ is large (compare with \cite{Sh:91}).
\bigskip

\noindent
(B99) \quad \underline{Equi-consistency results}, 
Notre Dame J. Formal Logic 26(1985), 178-188 (with L. Harrington).

There are two topics: 

\noindent
\textbf{MA and measurability and the Baire property}:  The following are
equi-consistent:
\mn
\begin{itemize}
    \item  ZFC + there is a weakly compact cardinal
\sn
    \item  ZFC + MA for every real $r$, $\aleph^{L[r]}_1 < \aleph_1$
\sn
    \item  ZFC + MA + every $\Delta^1_3$-set of reals is measurable
\sn
    \item  ZFC + MA + every $\Delta^1_3$-set of reals has the Baire property.
\end{itemize}
\mn
We can replace $\Delta^1_3$ by first order definable with a real and ordinal
parameter.
\bigskip

\noindent
\textbf{Reflection of stationary sets:}  The following are equi-consistent:
\mn
\begin{itemize}
\item  ZFC + there is a Mahlo cardinal
\sn
\item  ZFC + GCH + every 
stationary $S \subseteq S^2_0$ has an initial segment stationary.
\end{itemize}
\bigskip

\noindent
(B100) \quad \underline{Independence results}, JSL 45(1980), 563-573.

\noindent
\textbf{THE METHOD:}  We introduce the oracle chain condition which (for $\aleph_1$)
helps us to use iterated forcing of length $\aleph_2$, each step increasing
the set of reals, by enabling us to omit types (i.e. to promise no real
satisfying $\aleph_1$ conditions will be added later on).

We then introduce proper forcing, which is a very general condition
guaranteeing $\aleph_1$ is not collapsed and is preserved by countable
support iteration.

In fact we also introduce a blending of the two.

\noindent
\textbf{THE RESULTS:}  We prove the consistency of: $2^{\aleph_0} = \aleph_2$ and:
\mn
\begin{itemize}
\item  there is a universal linear order of power $\aleph_1$
\sn
\item  there are countable complete $T \subseteq T_1$, $T$ is not
$\aleph_0$-stable, not superstable, but still there is a unique $L(T)$-reduct
of a model of $T_1$ of power $\aleph_1$.
\end{itemize}
\mn
The results are complimentary to \cite[Ch.VII]{Sh:a}.  More see \cite{Sh:262}.
\bigskip

\noindent
(B101) \quad \underline{The theorems of Beth and Craig in 
abstract model theory. II. Compact}\\ \underline{Logics}, 
Archive Math Logik 21(1980), 13-36 (with J.A. Makowski).

Various compact logics such as stationary logic, positive logic, logics with
various cardinality quantifiers and cofinality quantifiers are studied.
Counterexmples to the theorems of Beth and Craig are given.  Back and forth
arguments are studied for the first two logics, transfer theorems presented
for positive logic and a new compactness proof for the cofinality quantifier
is given.
\bigskip

\noindent
(B102) \quad \underline{Forcing with stable posets}, JSL 47(1982), 37-42 (with U. Abraham).
\bigskip

\noindent
(B103) \quad \underline{On partitions of the real line}, 
IJM 32(1979), 299-304 (with D.H. Fremlin).

We proved that the real line is not necessarily the disjoint union of
$\aleph_1$ non-empty $G_\delta$-sets.  
This follows from known results and: if the real
line can be partitioned into $\kappa$ sets which are $G_\delta$-sets
($\kappa$ uncountable) then the real line can be covered by $\kappa$ nowhere
dense closed sets.
\bigskip

\noindent
(B104) \underline{The $\aleph_2$-Souslin hypothesis}, 
Trans. AMS 264(1981), 411-419 (with R. Laver).

We prove, assuming con(ZFC + there is a weakly compact cardinal) the
consistency of ``$\mathrm{ZFC} + \mathrm{CH} + 2^{\aleph_1} = 2^{\aleph_2} = \aleph_{3}$
+ there is no $\aleph_2$-Souslin tree."
\bigskip

\noindent
(B105) \quad \underline{On uncountable abelian groups}, IJM 32(1979), 311.

In \S1 we characterize the Whitehead groups of power $< 2^{\aleph_0}$,
assuming Martin Axiom: they are the $\aleph_1$-free groups satisfying
possibility II or III from \cite{Sh:44}; and, equivalently, they are
$\aleph_1$-coseparable or equivalently $\Ext(-,\bbZ_\omega)=0$.

In \S2 we construct an $\aleph_1$-free group satisfying possibility II which
is not strongly $\aleph_1$-free.  Hence $MA + 2^{\aleph_0} > \aleph_1$ 
implies there is a Whitehead group which is not strongly $\aleph_1$-free.

We also prove (assuming $\bfV = \bfL$ or even $2^{\aleph_0} < 2^{\aleph_1}$) that
there is a strongly $\aleph_1$-free, separable, not
$\aleph_1$-separable group of cardinality $\aleph_1$.  
At last we construct an $\aleph_2$-free (hence
strongly $\aleph_1$-free) non-separable, non-Whitehead group of cardinality
$2^{\aleph_1}$.  The method is using a tree which imitates
$\diamondsuit_{\aleph_1}$ in a sense.

In \S3 we deal with hereditarily separable groups.  If $\bfV = \bfL$ (or every
stationary subset of a regular cardinal is not small they are just the free
groups).  (This strengthens the theorem: if $\bfV = \bfL$ every Whitehead group is
free.)  But $MA + 2^{\aleph_0} > \aleph_1$ implies there are non-Whitehead,
hereditarily separable groups of cardinality.  We also prove, assuming
$2^{\aleph_0} < 2^{\aleph_1}$, that any hereditarily separable group is
strongly $\aleph_1$-free (a little more, in fact).
\bigskip

\noindent
CORRECTIONS: Omit 3.3,(iii) and 3.6.
\bigskip

\noindent
(B106) \quad \underline{Martin axiom does not imply that 
every two $\aleph_1$-dense sets of reals are}\\ \underline{isomorphic}, 
IJM 38(1981), 161-176 (with U. Abraham).

\textbf{RESULTS:}  We prove the consistency of the following with
ZFC + ``$2^{\aleph_0} = \aleph_2$" (assuming con(ZFC), of course). 

\noindent
1) There are two $\aleph_1$-dense sets of reals which cannot be made
isomorphic by any c.c.c. forcing notion (so \wilog\  MA holds). 

\noindent
2) There is an $\aleph_1$-dense set of reals such that any $f:A \rightarrow
A$ is order preserving on some uncountable $A' \subseteq A$. 

\noindent
3) For every $A \subseteq R$ of power $\aleph_1$, any one-to-one function
$f:A \rightarrow R$ is the union of countably many monotonic
functions. 

\noindent
4) MA + there are entangled sets of reals (defined there).

\textbf{METHOD:} A basic fact is that in order to make two $\aleph_1$-dense subsets
of $R$ isomorphic, there is a reasonably canonical forcing notion doing it,
if we are given a small enough closed unbounded subset as a parameter.
\bigskip

\noindent
(B108) \quad \underline{On successor of singular cardinals, study 
in logic and the foundation of}\\ \underline{Math}, 
Vol. 97(ed. boffa, Van dalen and McAllon) No. Holland, 1979, 357-386.

Let us for simplicity assume $\lambda = \aleph_{\omega +1}$, $\aleph_\omega$
strong limit.  We find a set 
$S^* \subseteq \lambda$, determine modulo the
closed unbounded filter, prove the equivalence of some definitions and:
\mn
\begin{itemize}
\item  if $P$ is $\aleph_n$-complete forcing notion, 
$S_1 \subseteq \{\delta < \lambda,\cf(\delta) < \aleph_n,\delta \in
S^*\}$ is stationary, then $S_1$ remains stationary after forcing with $P$.
However, there is an $\aleph_n$-complete forcing notion which makes $S^*$
non-stationary.
\sn
\item  $S^*$ has no initial segment $\delta$ stationary if GCH holds
\sn
\item  $S^* = \varnothing$ (i.e. not stationary) \Iff \, a weak version of
squares holds
\sn
\item  it is consistent (with ZFC + GCH) that $S^*$ is stationary (if
supercompact cardinals are)(the large cardinals are necessary).
\end{itemize}
\mn
Using (2) and (3) we get a proof, using different methods according to 
whether $S^*$ is stationary or not, of

COROLLARY:  ($\aleph_\omega$ strong limit) there is an 
$\aleph_{\omega+1}$-free non-free group of power $\aleph_{\omega +1}$
(similarly for transversal).  For higher $\lambda$ we get less information.

Donder noted that as if there is a special Aronszajn on $\aleph_{\omega +1}$
then the weak squares (from 3) holds we can (from (4), (3)) deduce:
ZFC + GCH is consistent with ``there is no special Aronszajn tree on
$\aleph_{\omega +1}$" (if a supercompact is consistent).
\bigskip

\noindent
(B109) \quad \underline{Infinite games and reduced powers}, 
AML 20(1981), 77-108 (with W. Hodges).

By \cite{Sh:13}, if $A,B$ are elementarily equivalent structures then there
are ultrafilters $D,E$ such that $A^I/D \cong B^J/E$.  The paper proves
analogues of this theorem where the language can be infinitary and $D,E$
are $\kappa$-complete for $\kappa$ a regular cardinal. 

\noindent
1) Let $\mathrm{PL}_\kappa$ be the language of prenex formulas of $L_{\kappa,\kappa}$
with arbitrary well-ordered game quantifiers of length $< \kappa$.  If
$A,B$ have mutually consistent $PL_\kappa$-theories and $\kappa$ is strongly
compact then $A,B$ have isomorphic limit ultrapowers where the filter
involved are $\kappa$-complete.

\noindent
2) Let $\mathrm{PH}_\kappa$ be the language of Horn formulas in $PL_\kappa$.  
If for every regular $\kappa$, $A$ and $B$ have mutually consistent 
$\mathrm{PH}_\kappa$-theories then $A,B$ have isomorphic limit reduced powers 
with $\kappa$-complete filters.

\noindent
3) If there is a proper class of measurable cardinals, then it is consistent
that (1), (2) hold with ``limit" deleted; this uses game-theoretic arguments
of Galvin and Laver.

\noindent
4) The natural analogue of (2) for classes of structures and limit reduced
products also holds and gives interpolation and preservation theorems for
$\mathrm{PH}_\kappa$ which had been derived proof-theoretically by Hodges.

\noindent
5) Necessary and sufficient conditions are given for a structure $A$ to have
a limit reduced power with $\kappa$-complete filters in which a given 
sentence is true; as a corollary, when $\kappa > \omega$.  The amalgamation
property fails for $\mathrm{PH}_\kappa$-elementary extensions.

\noindent
6) Under various conditions, examples are given of sentences which are
preserved in reduced products over $\kappa$-complete filters but are not
logically equivalent to sentences in $\mathrm{PH}_\kappa$, also answering a related
question of Kueker.

\noindent
7) A non-standard infinitary language is described which has good
interpolation and Feferman-Vaught properties.
\bigskip

\noindent
(B110) \quad \underline{Better quasi-order for uncountable cardinals}, IJM 42(1982), 177-226.

We generalize Nash-Williams theory on better quasi-order to uncountable
cardinals.

Let the well-ordering number of a quasi-order $Q$ be the minimal $\lambda$
such that for any $q_i \in Q$ ($i < \lambda$) for some $i < j$ we have
$q_i \le q_j$, and then we say $Q$ is $\lambda$-well ordered.  We define some 
variants of $\lambda$-better quasi-order, and prove for them preservation 
theorems under the known infinitary operations, (like Seq, ${\clP}$).  We 
then compute the well ordering number of some classes and of the result of 
the infinitary operations on arbitrary quasi-order. 

The main results are:
\mn
\begin{itemize}
\item  the well-ordering number of the class of graphs ordered by embeddings
(taking edges to edges rather than to paths (or disjoint paths)) 
is the first beautiful cardinal $\kappa_0$ (it is a mild large
cardinal, somewhat bigger than weakly compact, and still compatible
with $L$)
\sn
\item  also the well-ordering number of the class of models with one-place
predicates, under elementary embeddings, is the same
\sn
\item  the well-ordering number of $\{M:M \text{ a linear order and is the
union of } \le \lambda \text{ scattered ordered}\}\,\,(\lambda > \aleph_0)$,
ordered by embeddability, is the first beautiful cardinal $> \lambda$.
\end{itemize}
\bigskip

\noindent
(B112) \quad \underline{$S^-$ Forcing, a black box theorem for morasses with applications: super-}\\ \underline{Souslin trees and generalized MA}, 
Israel J. Math 43(1982), 185-224  (with L. Stanley).

\noindent
(A)  Quite a few consistency results can be gotten in the following way:
starting with $\bfV \models GCH$, we force an example of power $\aleph_2$ by
countable approximations.  So a countable approximation depends on a
countable set $\alpha$ of ordinals, so it is $\tau(\alpha)$ for some term
$\tau$, and the number of $\tau$-s is $2^{\aleph_0} = \aleph_1$.  We state
a black box principle, which says that many such proofs can be carried in
the universe $\bfV$ (if it satisfies the principle), moreover the proof that
the black box principle applies is the same as that the forcing works.

Of course, $\aleph_2$ can be replaced by higher cardinals. 

\noindent
(B)  We prove that the principle follows from $\bfV = \bfL$, really much less is
needed: the morass

\noindent
(C)  We got some consequences of the principle, among them the morass, so
they are equivalent.

\noindent
(D)  Another application is the existence of super-Souslin trees.  An
$\aleph_2$-super-Souslin tree $T$ is an $\aleph_2$-tree $T$, such that for any
$\alpha < \omega_2$, $\bar \alpha = \langle \alpha_n : \alpha < \omega \rangle$, 
and $\omega$ distinct members of $T_\alpha$, we have a function 
$$F_{\bar \alpha} : \Big\{ \langle \beta_b : n < \omega \rangle:\alpha_n \le \beta_n
\wedge (\exists \beta) \bigwedge\limits_{n} \beta_n \in T_\beta \Big\} \rightarrow
\omega_1$$ 
with $F_{\bar \alpha}(\beta) = F_{\bar \alpha}(\gamma)$ implying
$\beta_n,\gamma_n$ are comparable for infinitely many $n$'s.  (So it is
essentially Souslin as a special Aron tree is an Aron tree).

Now an $\aleph_2$-super-Souslin tree is $\aleph_2$-Souslin (ignoring possibly
finitely many branches), and it remains so in any extension not adding reals.

Hence: if $\CH + \SH_{\aleph_2}$ \then \, $\aleph_2$ is inaccessible in $\bfL$.
\bigskip

\noindent
(B114) \quad \underline{Isomorphism types of Aronszajn trees}, 
IJM 50(1985), 75-113 (with U. Abraham).

We investigate Aronszajn trees, under embeddability and isomorphism when a
restriction to a club (closed unbounded set of levels) $C$ is allowed, get
consequences of $2^{\aleph_0} < 2^{\aleph_1}$, and consistency results with
CH and with $2^{\aleph_0} = 2^{\aleph_1}$.  We also formulate a combinatorial
principle that follows from the weak diamond (which is equivalent to
$2^{\aleph_0} < 2^{\aleph_1}$) and hopefully will serve to obtain 
results like those in (A) below.

The results have variants speaking on Specker orders.

\noindent
(A) \textbf{Consequences} of $2^{\aleph_0} < 2^{\aleph_1}$ (\S1). 

If CH or even if $2^{\aleph_0} < 2^{\aleph_1}$ holds, then:
\mn
\begin{enumerate}
    \item[(a)]   there are $2^{\aleph_1}$ pairwise really different Aronszajn trees, i.e. not one of these trees is embeddable on a club into the other
\sn
    \item[(b)]   there is an Aronszajn tree $T$ such that for every closed unbounded 
    $C \subseteq \omega_1$, $T \rest C$ is rigid, (i.e. the only embedding of $T \rest C$ into $T \rest C$ is the identity, so $T$ is a \emph{really rigid} tree).
\end{enumerate}
\mn
We can combine (b) with (a):
\mn
\begin{enumerate}
\item[(c)]   for every Aronszajn tree $T^1$ there is an Aronszajn tree
$T^2$ such that $T^1$ is not embeddable into $T_2$ on a closed unbounded
set, i.e., there is no prime Aronszajn tree.
\end{enumerate}
\mn
The method is the combinatorial principle mentioned above.

\noindent
(B) \textbf{Consistency results with CH} (\S2).
\mn
\begin{enumerate}
\item[(a)]   GCH + the following is consistent:
\sn
\begin{enumerate}
\item[$(a_1)$]  there is a universal Aronszajn tree $T$ i.e., a tree
$T$ such that for every Aronszajn tree $T^*,T^* \restriction C$ is order
embeddable into $T \restriction C$ for some closed unbounded $C \subseteq
\omega_1$.  (The universal tree $T$ is a special Aronszajn tree), so in this
model there are no Souslin trees.  Thus $(a_1)$ is a strengthening of
Jensen's ``CH is consistent with Souslin's hypothesis," whose method we use.
\sn
\item[$(a_2)$]  Every two Aronszajn trees contain subtrees which are
isomorphic on a closed unbounded set.
\end{enumerate}
\sn
\item[(b)]   We have some consistency result with CH concerning Souslin
trees.  For example, CH + ``there is a Souslin tree" do not imply that there
are $2^{\aleph_1}$ really different Souslin trees.  Using this construction,
Shelah proved that the Malitz-Magidor logic is not necessarily compact.
\end{enumerate}
\bigskip

\noindent
\textbf{Problem}:  Does the existence of 
two really different Souslin trees follow from 
the existence of a Souslin tree?
\bigskip

\noindent
(C) \textbf{Consistency results with} $2^{\aleph_0} = 2^{\aleph_1}$.
\mn
\begin{enumerate}
\item[(a)]   Martin's axiom + $2^{\aleph_0} = \aleph_2$ + every two
Aronszajn trees are isomorphic on a closed unbounded set is consistent.
\sn
\item[(b)]    Martin's axiom + $2^{\aleph_0} > \aleph_1$ does not imply
that every two Aronszajn trees are isomorphic on a closed unbounded set.
\sn
\item[(c)]  Martin's axiom + $2^{\aleph_0} = \kappa$ + every two
Aronszajn trees are isomorphic on a closed unbounded set is consistent.
$\kappa$ here is ``any" regular cardinal such that 
$\kappa^{< \kappa}=\kappa$.
\end{enumerate}
\mn
(The difference between this item and (a) is that in (a) we get
$2^{\aleph_0} = \aleph_2$ and here $2^{\aleph_0}$ is as big as we want;
moreover, the proofs are different: in (a) we use proper forcing and in (c),
a technique of using generic reals).
\bigskip

\noindent
(B115) \quad \underline{Superstable fields and groups},  
AML 18(1980), 227-270 (with G. Cherlin).

The main object is to prove: for any infinite field $F$, $\Th(F)$ is
superstable \underline{iff} $F$ is algebraically closed.  This is done by
investigating connected groups ($G$ connected if it has no proper definable
subgroup of finite index) and indecomposable subgroups (same rank and
multiplicity one).  Also extend some known results for $\omega$-stable groups
with Morley rank $\le 3$ to superstable groups (i.e. the theory of the group
is superstable). 

An important ingredient is looking at definable groups in models of a 
stable theory and deal with local ranks $\rk(-,\Delta,\kappa)$, 
$\kappa \in \{2,\aleph_0\}$ for $\Delta$-s which are translation closed.
\bigskip

\noindent
(B116) \quad \underline{Positive results in abstract model theory: 
a theory of compact logics}, Ann. Pure Appl. Logic 25(1983), 263-299 (with J.A. Makowski).

We prove that compactness is equivalent to the amalgamation property provided
the occurency number of the logic is smaller than the first measurable
cardinal.  We also relate compactness to the existence of certain regular
ultrafilters associated with the logic and develop a general theory of
compactness and its consequences.  Continued in \cite{Sh:199}.
\bigskip

\noindent
(B117) \quad \underline{Combinatorial problems on trees: 
Partitions, $\Delta$-systems and large free}\\ \underline{subsets}, Ann. Pure Appl. Logic 33(1987), 43-81 (with M. Rubin).

We prove partition theorems on trees and generalize to a setting of trees
the theorems of Erd\H{o}s and Rado on $\Delta$-systems and the theorems of
Fodor and Hajnal on free sets.  For simplicity we ignore some of the results,
e.g. on trees with a non-fix splitting, or with $> \omega$ levels.  Let
$\mu$ be an infinite cardinal and $T_\mu$ be the tree of finite sequences of
ordinals $< \mu$, with the partial ordering of being an initial segment:
$\eta < \nu$ denotes that $\eta$ is an initial segment of $\nu$.  
A subtree of $T_\mu$ is a non-empty subset of $T_\mu$ closed under initial 
segments.  $T \le T_\mu$ means that $T$ is a subtree of $T_\mu$ and 
$\langle T, \le \rangle \cong T_\mu$.  

The following are extracts:
\begin{theorem}\label{b117.1}
(1)  \underline{A partition theorem}.  Suppose $\cf(\lambda) \ne
\cf(\mu)$, $F : T_\mu \rightarrow \lambda$ and for every branch $b$ of
$T_\mu$ we have $\sup(\{F(\alpha) : \alpha \in b\}) < \lambda$.
Then there is $T \le T_\mu$ such that 
$\sup(\{F(\alpha) : \alpha \in T\}) < \lambda$.
\end{theorem}

\begin{theorem}\label{b117.2}
(2) \underline{A theorem on large free subtrees}.  Let $\lambda^+ \le \mu$, 
$F : T_\mu \rightarrow P(T_\mu)$, for every branch $b$ of
$T_\mu:|\bigcup\{F(\alpha) : \alpha \in b\}| < \lambda$, and for every
$\alpha \in T_\mu$ and $\beta \in f(\alpha)$, $\beta \nless \alpha$.
Then there is $T \le T_\mu$ such that for every $\alpha,\beta \in T$,
$\beta \notin F(\alpha)$.

Let $P_\lambda(C)$ denote the ideal in $P(C)$ of all subsets of $C$ whose
power is less than $\lambda$.  Let $\cov(\mu,\lambda)$ mean that $\mu$ is
regular, $\lambda < \mu$, and for every $\kappa < \mu$ there is $D \subseteq
P_\lambda(\kappa)$ such that $|D| < \mu$ and $D$ generates the ideal
$P_\lambda(\kappa)$ of $P(\kappa)$.  Note that if for every 
$\kappa < \kappa^{< \lambda} < \cf(\mu) = \mu$, then $\cov(\mu,\lambda)$ holds.
Let $\alpha \wedge \beta$ denote the maximal common initial segment of 
$\alpha$ and $\beta$.
\end{theorem}

\begin{theorem}\label{b117.3}
(3) \underline{A theorem on $\Delta$-systems}.  Suppose
$\cov(\mu,\lambda)$ holds, $F : T_\mu \rightarrow P(C)$ and for every branch
$b$ of $T_\mu$ we have $|\bigcup\{F(\alpha) : \alpha \in b\}| < \lambda$, 
then there is $T \le T_\mu$ and a function $K : T \rightarrow P_\lambda(C)$ 
such that for every incomparable $\alpha, \beta \in T$, 
$F(\alpha) \cap F(\beta) \subseteq K(\alpha \wedge \beta)$.
\end{theorem}
\bigskip

\noindent
(B118) \quad \underline{On the expressibility hierarchy of 
Magidor-Malitz quantifiers}, JSL 48(1983), 542-557 (with M. Rubin).

We prove that the logics of Magidor Malitz and their generalization by Rubin
are distinct even for PC classes; during this we prove the existence of
forcing notions satisfying one variant of the countable chain condition but
not another and use the preservation of the satisfaction of sentences in
$\bbL^n$ (see below) under appropriate chain condition.

Let $M \models Q^nx_1 \ldots x_n \varphi(x_1 \ldots x_n)$ mean that there is
an uncountable subset $A$ of $|M|$ such that for every $a_1,\dotsc,a_n \in
A,M \models \varphi[a_1,\dotsc,a_n]$.
\bigskip

\begin{theorem}\label{b118.1}
(1)  1) $(\diamondsuit_{\aleph_1})$.  For every $n \in \omega$ the class 
$K_{n+1} = \{\langle A,R \rangle:\langle A,R \rangle \models Q^{n+1}
x_1 \ldots x_{n+1} R(x_1,\dotsc,x_{n+1})\}$ is not an $\aleph_0-PC$-class in
the logic $L^n$, obtained by closing first order logic under $Q^1,\dotsc,
Q^n$, i.e., for no countable $L^n$-theory $T,K_{n+1}$ is the class of reducts
of the models of $T$. 

\noindent
2) (CH) For every $n \in \omega$ there is a forcing notion $(P,\le)$
satisfying $(Q^ix_1 \ldots x_i)(x_1,\dotsc,x_i$ have no upper bound)
\Iff \, $i > n$.
\end{theorem}

\begin{theorem}\label{b118.2}
(2)
$(\diamondsuit_{\aleph_1})$.  Let $M \models (Q^E x,y) \varphi(x,y)$ mean that 
there is $A \subseteq |M|$ such that
$E_{A,\varphi} =: \{\langle a,b \rangle : a,b \in A$ and $M \models \varphi
[a,b]\}$ is an equivalence relation on $A$ with uncountably many equivalence
classes, and such that each equivalence class is uncountable.  Let
$K^E = \{ \langle A,R \rangle : \langle A,R \rangle \models (Q^Ex,y)R(x,y)\}$,
then $K^E$ is not an $\aleph_0-PC$-class in the logic gotten by closing first
order logic under the set of quantifiers $\{Q^n:n \in \omega\}$ which were
defined in Theorem 1.
\end{theorem}

A more general version of the second theorem is proved in the paper.
\bigskip

\noindent
(B119) \quad \underline{Iterated forcing and changing cofinalities}, IJM 40(1981), 1-32.
\bigskip

\noindent
(B120) \quad \underline{Free limits and Aronszajn trees}, IJM 38(1981) 33-38.

Our main result is the consistency of the ``ZFC + Souslin hypothesis + not
every Aronszajn tree is special."

For this end we introduce some variant of being special, investigate the
connections between them, and for one ``being $S$-special" ($S \subseteq
\omega_1$ stationary co-stationary) we prove that if we $S$-specialize all
Aronszajn trees (by iterated forcing), some ``bad" tree we specify does not
become special; as some related property of the forcing is preserved.

We use (as clearer though CS is O.K.) 
an iteration in which we take a kind of limit called free (it is even
bigger than inverse limit).
\bigskip

\noindent
(B121) \quad \underline{On the standard part of non-standard models 
of set theory}, JSL 48(1983), 33-38 (with M. Magidor and J. Stavi).

We characterize the ordinals $\alpha$ of uncountable cofinality such that
$\alpha$ is the standard part of a non-standard model of ZFC iff 
$\alpha$ is weakly compact in some sense; more exactly:
\mn
\begin{enumerate}
\item[$(*)$]   there is $\gamma > \alpha$ such that:
\sn
\begin{enumerate}
\item[(a)] for arbitrarily large $\beta < \gamma$, $\bfL_\beta$ is a model of ZFC
\sn
\item[(b)]  in $L_\gamma,\alpha$ has the tree property, i.e., if
whenever $T \in L_\gamma$, is a tree $\subseteq \alpha$, of height $\alpha$
and each level has $A$-cardinality $< \alpha$, and the function $\beta
\mapsto$ level of $\beta$ in $T$, is in $A$ \then \, $T$ has a branch of
length $\alpha$ in $A$.
\end{enumerate}
\end{enumerate}
\bigskip

\begin{remark}\label{b121}
We can replace $\bfL_\beta \models$ ZFC by $\bfL_\beta \models$ KP.
\end{remark}
\bigskip

\noindent
(B122) \quad \underline{On Fleissner diamond}, NDJFL, 22(1981), 29-35.

We prove that $ZFC + GCH + \diamondsuit^+_\lambda$ is consistent with 
(where $1 < \kappa < \lambda$, $\lambda$ regular $> \aleph_0)$: there are
$S_{\alpha,\beta} \subseteq \alpha$ (for $\alpha < \lambda,\beta < \kappa$) such
that:
\mn
\begin{itemize}
\item   $\big\{\{S_{\alpha,\beta}:\beta < \kappa\} : \alpha < \lambda\big\}$
is a diamond sequence, i.e., for every $A \subseteq X$ for stationarily many
$\alpha < \lambda$, $A \cap \alpha \in \{S_{\alpha,\beta}:\beta < \kappa\}$
\sn
\item   for no function $f:\lambda \rightarrow \kappa$ is $\langle
S_{\alpha,f(\alpha)}:\alpha < \lambda \rangle$ a diamond sequence; moreover,
even $\big\LL \{S_{\alpha,\beta} : \beta < \kappa\} \setminus 
\{S_{\alpha, f(\alpha)} : \alpha < \kappa\} \big\RR$ is not a diamond sequence.
\end{itemize}
\bigskip

\noindent
(B123) \quad \underline{Monadic theory of order and topology in ZFC}, AML 23(1982), 179-198 (with Yuri Gurevich).

True first-order arithmetic is interpreted in the monadic theory of the real
line, in the monadic theory of any non-modest short chain, in the monadic
theory of Cantor Discontinuum, in the monadic theory of any non-modest
metrizable space.  It was known that existence of such interpretations is
consistent with ZFC.
\bigskip

\noindent
(B124) \quad \underline{$\aleph_\omega$ may have a strong partition 
relation}, IJM 38(1981), 283-288.

We prove the consistency of ZFC + GCH with the following (assuming 
con(ZFC) + there are $\omega$ measurable cardinals):
\mn
\begin{enumerate}
\item[$(*)$]   for any function $f$ from finite subsets of $\aleph_\omega$
to $\omega$, there are pairwise disjoint $S_n \subseteq \aleph_\omega,
|S_n| = \aleph_n$ such that:
\newline
\qquad if $u,v$ are finite subsets of $\aleph_\omega$ and 
$(\forall n)\big[|S_n \cap u| = |S_n \cap v| \le n \big]$ \then \, $f(u) = f(v)$.
\end{enumerate}
\mn
Note that $(*)$ implies: every Banach space of power $\ge \aleph_\omega$
contains an unconditional (infinite) basis.
\bigskip

\noindent
(B125) \quad \underline{The consistency of $\Ext(G,\bbZ) = \bbQ$}, 
IJM 39(1981), 82.

For abelian groups, if $\bfV = \bfL$, $\Ext(G,\bbZ)$ cannot have cardinality
$\aleph_0$.  We show that GCH does not imply this (hence some complete
$\bbL_{\infty,\aleph_1}$ sentence has exactly $\aleph_0$ non-isomorphic models
of power $\aleph_1$, but not proved here).  (Compare [Sh:91]).

The proof continues in \cite{Sh:64}.
\bigskip

\noindent
(127) \quad \underline{On Boolean Algebras with no uncountably 
many pairwise comparable or}\\ \underline{uncomparable elements}, NDJFL 22(1981), 301-308.

We shall prove, assuming CH, the existency of a Boolean algebra of power
$\aleph_1$, having neither uncountable set of pairwise comparable elements
nor uncountable set of pairwise incomparable elements.

We finish commenting on generalizations to higher cardinals.
\bigskip

\noindent
(B128) \quad \underline{Uncountable constructions}, IJM 51(1985), 273-297.

We suggest a principle which unites many proofs from 
$\diamondsuit_{\aleph_1}$, and is similar to forcing and deduce from 
it several new results.  The paper is supposed to help non-set theorists to
build counterexamples by reducing the problems to questions on finitely
generated objects: (assuming $\diamondsuit_{\aleph_1}$).

\textbf{Boolean Algebras:} 

\noindent
1)  For pedagogical reasons we build a Rubin Boolean Algebra. 

\noindent
2)  There is a Boolean Algebra which is not 1-Rubin but among any $\aleph_1$
elements there are two comparable ones (even $a,b,c$ with $a \cap b = c$).
\mn

\textbf{E.C. Groups:}  For any countable e.c.g. and uncountable $K$, there is an
uncountable $H \equiv {}_{\infty,\omega}G$, such that $K$ cannot be 
embeddable into $H$.

\textbf{Banach spaces:}  There is a non-separable Banach space $B$ so that for any
non-separable subspace $B_1$, $B/B_1$ is separable.
\bigskip

\noindent
(B129) \quad \underline{The number of non-isomorphic 
models of cardinality $\lambda$, $\bbL_{\infty,\lambda}$ equivalent} 
\underline{to a fixed model}, NDJFL, 22(1981), 5-10.

We show that if $\bfV = \bfL$ and $\lambda$ is regular but not weakly compact, then
if $M$ is a model of cardinality $\lambda$ then 
$\{N/{\cong} : N \equiv_{\infty,\lambda} M,\ \|N\| = \lambda\}$ 
has either one number or $2^\lambda$ members.

We do not rely on the full power of $\bfV = \bfL$ but use a combinatorial
consequence for rgular non-weakly compact.  The use of 
some set theoretical assumption stronger then GCH
is necessary as in \cite{Sh:125} we construct a model of ZFC + GCH where a
sentence $\psi \in \bbL_{\infty,\omega_1}$ has $\aleph_0$ non-isomorphic models
of cardinality $\aleph_1$.  Also the restriction on $\lambda$ to be not
weakly compact is not a coincidence; see \cite{Sh:133}.
\bigskip

\noindent
(B130) \quad \underline{Stability over a predicate}, NDJFL, 16(1985), 361-376 (with A. Pillay).

This is a beginning of a theory suppose to measure how much $M \restriction
P$ determines $M$, for this we look at $\{M:M \restriction P = N\}$ 
and count the number of isomorphism types (over $N$).  In the cases 
the number is not big, we have
parallels to some of the theorems on stable theories.
\bigskip

\noindent
(B133) \quad \underline{On the number of non-isomorphic 
models in $\bbL_{\infty,\kappa}$ when $\kappa$ is weakly}\\ \underline{compact}, NDJFL, 23(1982) 21-26.

We construct for $\kappa$ weakly compact a model $M$ of cardinality $\kappa$
which has exactly $\theta$ non-isomorphic models $M \equiv_{\infty,\kappa} M$
of cardinality $\kappa$, for any $\theta \le \kappa$ we want.  This 
together with \cite{Sh:129} completes the answer of this problem when $\bfV = \bfL$.
\bigskip

\noindent
(B135) \quad \underline{Rigid homogeneous chains}, Math. Proc. of Cambridge Philo. Soc. 87(1981), 7-17 (with A.M.W. Glass, Yuri Gurevich and W.C. Holland).

A chain $C$ is \emph{rigidly homogeneous} if 
for any two points $a,b$ of $C$ there is a unique automorphism $f$ of 
$C$ such that $f(a) = b$.  The group of automorphisms of a rigidly 
homogeneous chain $C$ is isomorphic to a subgroup of the additive 
group of reals.  We classify rigidly homogeneous chains by elementary 
properties of their automorphism groups.

The chain of integers is the only countable rigidly homogeneous chain.
The cardinality of every rigidly homogeneous chain is $\le 2^{\aleph_0}$.
Ohkuma constructed rigidly homogeneous chains of cardinality
$2^{\aleph_0}$.  We prove consistency of existence as well as consistency of
non-existence of rigidly homogeneous chains of uncountable cardinality
$< 2^{\aleph_0}$.
\bigskip

\noindent
(B136) \quad \underline{Construction of many complicated 
uncountable structures and Boolean}\\ \underline{Algebras}, IJM, 45(1983), 100-146.

This paper has three aims:

\noindent
1) To make the results of \cite[Ch.VIII]{Sh:a} on constructing models more
available for applications, by separating the combinatorial parts: for
this we have to set up a suitable framework.  We believe it is useful for
proving things like: there are many non-isomorphic structures, there are
many structures no one embeddable into the other, there are rigid structures
or undecomposable ones or structures with few endomorphisms, etc.

There are few kinds of structures (mainly trees) which served as index sets,
and we have many complicated such structures.  In an application we have in
some sense to interpret those structures in the one we are interested in. 

\noindent
2) A second aim is to strengthen the results of \cite[Ch.VIII]{Sh:a}.  In
particular we mainly are interested there in showing that there are many
non-isomorphic models of an unsuperstable theory.  We got many times results
on the number of models not elementarily embeddable in each other, being a
side benefit.  Here we consider the later case in more detail, and in a 
somewhat stronger combinatorial results needed later.

We also consider some more kinds of index structures 
(mainly a tree of pairs).

\noindent
3) We apply the previous results to problems on Boolean Algebras 
from the list of van Douwen, Monk, and Rubin.  

In particular
\mn
\begin{enumerate}
\item[(a)]   if $\lambda = \lambda^{\aleph_0}$ there is a complete
Boolean Algebra of power $\lambda$ satisfying the c.c.c. with no non-trivial
one-to-one endomorphism
\sn
\item[(b)]  if we wave ``complete" then is such $B$ of any power
$\lambda > \aleph_0$
\sn
\item[(c)]   in any $\lambda > \aleph_0$ there is a Bonnet-rigid 
Boolean Algebra.
\end{enumerate}
\mn
4) To solve some problems from that list, unconnected to 1) and 2):
\mn
\begin{enumerate}
\item[(a)]   if a Boolean Algebra $B$ has no dense subset of power
$< \lambda$ then it has a $\lambda$ pairwise incomparable element
\sn
\item[(b)]   $(\diamondsuit_{\aleph_1})$ there is a Boolean Algebra of
power $\aleph_1$ satisfying many rigidity conditions: Bonnet rigid endo-rigid
indecomposable.
\end{enumerate}
\mn
The idea of this proof is to generalize Keisler criterion for omitting types
(specific to atomless Boolean Algebras) to the case an ultrafilter is given
in each countable stage of the construction. 
l
Continued in \cite{Sh:e}.
\bigskip

\noindent
(B137) \quad \underline{The Singular cardinals problem: 
independence results}, London Math. Soc. Lecture Note Ser. 87(1983) 116-134.

Assuming the consistency of a supercompact cardinal, we prove the consistency
of:
\mn
\begin{itemize}
\item  $\aleph_\omega$ strong limit $2^{\aleph_\omega} = \aleph_{\alpha +1}$,
$\alpha < \omega_1$ arbitrary;
\sn
\item  $\aleph_{\omega_1}$ strong limit, $2^{\aleph_{\omega_1}} =
\aleph_{\alpha +1}$, $\alpha < \omega_2$ arbitrary;
\sn
\item  $\aleph_\delta$ strong limit, $\cf(\delta) = 
\aleph_0$, $2^{\aleph_\delta}$ arbitrarily large before the first 
inaccessible cardinal; for $\aleph_\delta$ ``large enough".
\end{itemize}
\mn
We start with $\kappa$ supercompact, e.g. an $\omega$-sequence
$\langle \kappa_n : n < \omega \rangle$ through it and collapse to make
$\kappa$ to $\aleph_\omega$.  The new point is that we change our mind 
on what is $\bfV$ --- it becomes larger when we determine $\kappa_n$ for bigger
$n$.
\bigskip

\noindent
(B138) \quad \underline{On the structure of $\Ext(A,\bbZ)$ in ZFC}, 
JSL, 50(1985), 302-315 (with G. Sageev).

We prove:
\mn
\begin{itemize}
\item  (GCH) for any $\aleph_0 < \kappa < \aleph_\omega$, and cardinals
$\nu_p \le \kappa^+ $ (for $0 < p < \omega$ prime) there is a $\kappa$-free abelian
group of power $\kappa$, $\nu_0(G) = \kappa^+$, $\nu_p(G) = \nu_p$ (see
\cite{Sh:91}).
\end{itemize}
\bigskip

\noindent
(B139) \quad \underline{On the number of non-conjugate subgroups}, 
Algebra Universalis, 16(1983), 131-146.

Let $G$ a given group of cardinality $\lambda \, (\lambda > \aleph_0)$ and
$\mathrm{nc}(G)$ the number of non-conjugate subgroups of $G$, so we are interested
in lower bound of the number of equivalence relations.

We prove for many uncountable cardinals $\lambda$ that if $|G| = \lambda$
then $\mathrm{nc}(G) \ge \lambda$, (e.g. when GCH holds). 

This is proved for $\aleph_0 < \lambda < 2^{\aleph_0}$ (for $\lambda =
\aleph_1$ assuming $2^{\aleph_0} = \aleph_1$ by \cite{Sh:69} this is best
possible); for $\lambda$ such that $\forall \mu[\lambda \ne 2^\mu]$.  For
$\lambda = 2^\mu$ and $\mu$ not regular non-strong limit and $\mathrm{Ded}(\mu) > 
\lambda$.  It is an open problem if the theorem is true for other cardinals.
Continued in \cite{Sh:192}.
\bigskip

\noindent
(B140) \quad \underline{On endo-rigid $\aleph_1$-strongly free 
abelian groups of power $\aleph_1$}, IJM, 40(1981), 291-295.

We prove that if $2^{\aleph_0} < 2^{\aleph_1}$ then there is an abelian group
of power $\aleph_1$, which is strongly $\aleph_1$-free but endo-rigid
where:
\mn
\begin{itemize}
\item   $G$ is endo-rigid iff every homomorphism $h:G \rightarrow G$
has the form $h(x) = n(x)$ for every $x$; (for some integer $n$).
\end{itemize}
\bigskip

\noindent
(B141) \quad \underline{The monadic theory of $\omega_2$}, JSL 48(1983), 387-398 
(with Y. Gurevich and M. Magidor).

Assume consistency of ZFC plus existence of a weakly compact cardinal.  Then
\mn
\begin{enumerate}
\item[(i)]   for every $S \subseteq \omega$, ZFC plus ``$S$ and the
monadic theory of $\omega_2$ are recursive in each other" is consistent,
and
\sn
\item[(ii)]   ZFC plus ``the full second-order theory of $\omega_2$ is
interpretable in the monadic theory of $\omega_2$" is consistent.
\end{enumerate}
\bigskip

\noindent
(B142) \quad \underline{The structure of saturated free algebras}, 
Algebra Universalis, 17(1983), 191-199 (with J. Baldwin).

Let $V$ be a variety with countable similarity type.  Suppose $M$ is an
uncountable algebra in $V$ which is both free and saturated. 

\begin{lemma}\label{b142.0}
$\Th(M)$ is $\aleph_0$-stable.
\end{lemma} 
\medskip

\noindent
\textbf{Theorem 142.1.}\label{b142.1}  There exists a finite set $q_1,q_2,\dotsc,q_\kappa$ of types (over
finite sets) each with weight one such that there exist infinite sets of
indiscernibles $Y_1,\dotsc,Y_n$ each $Y_i$ based on some $q_i$ such that $M$
is generated as an algebra by $Y \cup \ldots \cup Y_n$.
\bigskip

\noindent
(B143) \quad \underline{The monadic theory and the 
``next world"}, IJM 49(1984), 55-68 (with Yuri Gurevich).

Let $R$ be the real line in a model $\bfV$ of ZFC and $B$ be the Boolean Algebra
of regular open subsets of $R$ in $\bfV$.  Let $V^B$ be the corresponding 
Boolean valued model of ZFC.
\medskip

\noindent
\textbf{Theorem 143.1.}\label{b143.1}  There is an algorithm interpreting the full second-order
$\bfV^B$-theory of $\aleph_0$ in the monadic $\bfV$-theory of $R$.
\medskip

\noindent
\textbf{Theorem 143.2.}\label{b143.2}  There is an algorithm interpreting the full second-order
$\bfV^B$-theory of $2^{\aleph_0}$ in the monadic $\bfV$-theory of $R$ if $\bfV$
satisfies the Continuum Hypothesis. 

Theorems 1 and 2 are results of applications of a general interpretation
theorem.  

Continued in \cite{Sh:284a}.
\bigskip

\noindent
(B217) \quad  \underline{There are Noetherian domains in every cardinality with free additive}\\ \underline{groups}, Abstracts Amer. Math. Soc. 7(1986), 369 (with G. Sageev).

(Based on Notices AMS 86T-03-269)
\bigskip

\begin{theorem}\label{217.1}
There are Noetherian rings (in fact domains) with a free
additive group, in every infinite cardinality.
\end{theorem}

\begin{remark}\label{217.2}
1) For $\aleph_1$ this was proved by O'Neill. 

\noindent
2) The work was done in Sept., '83. 

\noindent
3) We thank Fuchs for suggesting to us the problem.
\end{remark}

\begin{proof}
\underline{Sketch of Proof}  

Let $\bbZ$ be the ring of integers, $X$ a set of distinct
variables, $\bbZ[X]$ the ring of polynomials over $\bbZ$, $\bbZ(X)$ its field of
quotients, and $R_X$ the additive subgroup of $\bbZ(X)$ generated by
$$\big\{p/q : p,q \in \bbZ[X],\ p \text{ not divisible (non-trivially)
by any integer}\big\} \subseteq \bbZ(X).$$  
It is known that $R_X$ is a Noetherian domain.
Let, for a ring $R$, $R^+$ be its additive group.  For $Y \subseteq X$ we can
define $\bbZ[Y]$, $\bbZ(Y)$, $R_Y$ similarly.
\end{proof}

\begin{lemma}\label{b217.3}
1) $R^+_X$ is a free abelian group. 

\noindent
2) If $n \ge 0$, $Y \subseteq X$, $x(1),\dotsc,x(n) \in X \setminus Y$ 
pairwise distinct, $W = \{x(1),\dotsc,x(n)\}$, $W(\ell) = W - \{x(\ell)\}$ \underline{then} $R^+_{W \cup Y}\Big/\sum\limits^{n}_{\ell =1} R^+_{W(\ell) \cup Y}$ is a free abelian group.
\end{lemma}

\begin{proof}
1) Follows by 2) for $n=0,Y=X$. 

\noindent
2) This is phrased because it is the natural way to prove 1) by induction
on $|Y|$, for all $n$ simultaneously (a degenerated case of
\cite{Sh:87a}).   If $|Y| > \aleph_0$, let
$Y = \{y(\alpha):\alpha < \lambda\}$ with no repetitions, so $\lambda = |Y|$, 
$Y_\alpha = \{y(i) : i < \alpha\}$.  It suffices for each $\alpha < \lambda$ to 
prove that $G_\alpha =: \big(R^+_{Y_{\alpha+1}} + \sum\limits^{n}_{\ell =1}
R^+_{Y \cup W(\ell)}\big) \Big/ \big(\sum\limits^{n}_{\ell =1}
R^+_{Y_{\alpha+1} \cup W(\ell)} + R^+_{Y_\alpha \cup W} \big)$ is free.  
We now show that $G_\alpha$ is
isomorphic to $G'_\alpha = R^+_{Y_{\alpha+1}} \Big/ \big(\sum\limits^{n}_{\ell =1}
R^+_{Y_{\alpha+1} \cup W(\ell)} + R^+_{Y_\alpha \cup W} \big)$.
For this it is enough to show $\bigl(\sum\limits^{n}_{\ell=1}
R^+_{Y_\alpha \cup W(\ell)} \bigr) \cap
R^+_{Y_{\alpha +1}} = \sum\limits^{n}_{\ell =1} R^+_{Y_{\alpha +1} \cup
W(\ell)}$, as the right side is included in the left side trivially we have to
show $$\sum^{n}_{\ell=1} {\frac{p_\ell}{q_\ell}} \in 
\sum\limits^{n}_{\ell =1} R^+_{Y_{\alpha +1} \cup W(\ell)}$$ if
${\frac{p_\ell}{q_\ell}} \in R^+_{Y_1 \cup W(\ell)}$ and
$\sum\limits^{n}_{\ell =1} {\frac{p_\ell}{q_\ell}} \in 
R^+_{Y_{\alpha +1} \cup W}$ which is easy by projections). 
But $G'_\alpha$ is free by induction hypothesis.  
\end{proof}
\bigskip

The next claim completes the case ``$y$ countable".
\begin{claim}\label{217.4}
If $Y \cup \{x(1),\dotsc,x(n)\} \subseteq X$, $x(\ell) \in X 
\setminus Y$ distinct, $G = R^+_{W \cup Y}$, $I = \sum\limits_{i} I_i$, 
$I_i = R^+_{W(i) \cup Y}$ \then \, $G/I$ is free, when $Y$ is countable.
\end{claim}

\begin{proof}
It suffices to prove:
\mn
\begin{enumerate}
\item[(a)]   $G/I$ is torsion free
\sn
\item[(b)]   if $a_1,\dotsc,a_k \in G/I$ are independent, \then \,
$$\Big\{m \in \bbZ^+ : \text{ there are } \langle q_1,\dotsc,q_k \rangle \in L \text{ such that } m\text{ divides} \sum\limits_{i=1}^k q_i a_i \text{ in } G/I \Big\}$$ 
is finite, where 
$$L = \big\{\langle q_1,\dotsc,q_k \rangle : q_i \in \bbZ\text{, not all 
zero with no common divisor}\big\}.$$
\end{enumerate}
\mn
Let $x_1(q) \in X$ for $q=1,\dotsc,n$ be new distinct variable and let
$V = \{x_1(1),\dotsc,x_1(n)\}$.  For $u \subseteq \{1,\dotsc,n\}$ let us
define $h_u:R_{V \cup W \cup Y} \rightarrow R_{V \cup Y}$ an isomorphism 
$h_u(y) = y$ for $y \in Y$, $h_u(x(q)) = x_1(q)$ if $q \in u$, $h_u(x(q)) = x(q)$ 
if $q \notin u$.  So let $a_1 + I,\dotsc,a_k + I$ be independent. 

Suppose $\langle q_1,\dotsc,q_k \rangle \in L$, $m_0 m_1 \in \bbZ \setminus \{0\}$, 
$m_0 m_1$ divides $\sum\limits_{i} m_0 q_i a_i + I$.  So for some $s \in 
R_{W \cup Y}$ and $p_\ell \in I_\ell$ for $\ell = 1,\dotsc,n$ we have:
$\sum\limits_{i} m_0 q_i a_i = m_0 m_1 s + \sum\limits_{\ell =1,\dotsc,n} 
p_\ell$.  

Let $u$ vary on subsets of $\{1,\dotsc,n\}$, $b_u = \sum\limits_{u}(-1) {}^{|u|} h_u(a_\ell) \in R_{V \cup W \cup Y}$, so
$$\sum\limits_{i} m_0 q_i b_i = \sum\limits_{u}\sum_i m_0 q_i(h_u(a_i)) = 
m_0 m_1 \sum\limits_{u} h_u(s) + \sum\limits_{\ell =1,\dotsc,n} 
\sum\limits_{u} h_u(p_\ell).$$ 

However for each $\ell =1,\dotsc,n$ we have $\sum\limits_{u} h_u(p_\ell)$ is 
zero (as $x(\ell)$ does not appear in it).

So $\sum\limits_{i} m_0 q_i b_i$ is divisible by $m_0 m_0$ in
$R^+_{V \cup W \cup Y}$.  As $R^+_{V \cup W \cup Y}$ is free, it 
suffices to prove $\{b_i:i=1,\dotsc,k\}$ is
independent, equivalently they are linearly independent (over the rationals)
in $Z(Y \cup W \cup V)$.  But, if not, we can substitute suitable numbers for
$x_1(1),\dotsc,x_1(n)$ and get contradiction to 
``$\{a_i + I:i = 1,\dotsc,n\}$ is independent."

That is let $R'$ be a subring of $R_{V \cup W \cup Y}$ generated by
$\bbZ[X] \cup \{\frac{1}{q_1},\dotsc,\frac{1}{q_m}\}$ for some $m,q,\dotsc,
q_\ell \in \bbZ[X]$ such that $h_u(a_i) \in R'$.  Let $g$ be a homomorphism from
$R'$ to $R_{W \cup Y}$ which is the identity on $R_{W \cup Y}$ and maps each
$x_1(q)$ to an integer (so we require from $\big\langle g(x_i(q)):q=1,\dotsc,n
\big\rangle$ to make some finitely many polynomials over the integers nonzero
which is possible).  Now $\ell \in u \subseteq \{1,\dotsc,n\} \Rightarrow
h_u(a_i) \in I_\ell$.  So it is enough to show that $\langle g(b_i):i=1,
\dotsc,k \rangle$ is linearly independent.  But $g(b_i) = \sum\limits_{u}
g(h_u(a_i)) \in gh_\varnothing(a_i) +I = g(a_i) + I = a_i + I$.
\end{proof}
\bigskip

\noindent
(C1) \quad \underline{On Shelah compactness of cardinal}, 
IJM, 31(1978), pgs.34-56 and pg.394 (with S. Ben-David)

We deal with the compactness property of cardinals presented by
Shelah, who proved a compactness theorem for singular cardinals.  We
improve that result in eliminating axiom I there and show a new
application of that theorem together with a straightforward proof of
it for the special case discussed.  We discuss compactness for regular
cardinals and show some independence results: one of them, a part of
which is due to A. Litman, is the independence from ZFC + GCH of the
gap---one two cardinal problem for singular cardinals.
\bigskip

\noindent
(C2) \quad \underline{The Shelah $P$-point independence results}, IJM, 43(1982), 28-48. [C2], (E. Wimmers).

Present the proof of some forcing extension there is no $P$-point.  We
start with $\bfV \models ``2^{\aleph_0} = \aleph_1 + 2^{\aleph_1} =
\aleph^+_2$." Consider a list $\langle D_\alpha:\alpha <
\aleph_2\rangle$ of the filter $D$ on $\bbN$ such that 
${\clP}(\bbN)/D$ satisfies the c.c.c. (and every co-finite subset of
$\bbN$ belongs).  We then force a close enough subset of the CS
product of $\prod\limits_{\alpha} \bbQ_{D_\alpha}$ for appropriate
$\bbQ_{D_\alpha}$.
\bigskip

\noindent
(C3) \quad \underline{The uniformization property for $\aleph_2$}, 
IJM, 36(1980), 248-256. [C3], (C. Steinhorn and J. King).

We deal with the following problem: let 
$S^2_1 = \{\delta < \aleph_2 : \cf(\delta) = \aleph_1\}$, $\eta_\delta$ 
an $\omega_1$-sequence converging to $\delta$, 
$\Phi = \langle \eta_\delta:\eta_\delta \in S^2_1 \rangle$.  
Assuming GCH, can $(\Phi,\aleph_1)$ have the uniformization
property, at least for some $\Phi$ (see \cite{Sh:64})?  The answer is 
positive, so not only GCH $\not\rightarrow \diamondsuit_{S^2_1}$, 
but also GCH is consistent with $S^2_1$ being small.

We start with $V \models GCH$ and use iteration of length $\omega_3$, with
support of power $\le \aleph_1$, of the obvious forcings introducing a
uniformizing functions each condition being an initial segment.  The proof
is like \cite{Sh:64} taking $N \prec ({\cH}(\aleph_1),E)$, closed 
under countable sequences.  But trying to build trees as there we run 
into two problems:
\mn
\begin{enumerate}
\item[(a)] even if we get a tree of candidates for being a conditions,
we cannot find $\aleph_1$-closed $N_1$ with $N_1 \cap \omega_2 \notin S^2_1$
(if we replace $S^2_1$ by $S \subseteq S^2_1$, $S^2_1 \setminus S$ stationary,
our life would be much easier).
\end{enumerate}
\bigskip

\bibliographystyle{amsalpha}
\bibliography{shlhetal}

\end{document}